\documentclass[]{interact}

\usepackage{epstopdf}% To incorporate .eps illustrations using PDFLaTeX, etc.
\usepackage[caption=false]{subfig}% Support for small, `sub' figures and tables
\usepackage[numbers,sort&compress]{natbib}% Citation support using natbib.sty
\bibpunct[, ]{[}{]}{,}{n}{,}{,}% Citation support using natbib.sty
\makeatletter% @ becomes a letter
\def\NAT@def@citea{\def\@citea{\NAT@separator}}% Suppress spaces between citations using natbib.sty
\makeatother% @ becomes a symbol again

\usepackage{amsthm}%������֤���໷��������
\usepackage[usenames]{color}
\usepackage{xspace,colortbl}
\usepackage{latexsym,url,amscd,amssymb}
\usepackage{amsmath}
\usepackage{hyperref}
\usepackage{cases}
\usepackage{caption}
\usepackage{algorithm}
\usepackage{algpseudocode}
\usepackage{listings}
\usepackage{rotating}
\usepackage{framed}
\usepackage{amssymb}
\usepackage{float}
\usepackage{amsmath}\numberwithin{equation}{section}
\usepackage{multirow}
\usepackage{makecell}
\usepackage{booktabs}
\usepackage[utf8]{inputenc}  % ??UTF-8??
\usepackage[T1]{fontenc}

\DeclareMathOperator*{\argmin}{arg\,min}

\DeclareMathOperator*{\dist}{dist}

\def\M{{\mathcal M}}

\newtheorem{lemma}{Lemma}[section]
\newtheorem{definition}{Definition}[section]

\newtheorem{theorem}{Theorem}[section]

\renewcommand{\thesection}{\arabic{section}}
\renewcommand{\thesubsection}{\arabic{section}.\arabic{subsection}}

\begin{document}

\title{Bilevel Programming Approach for Image Restoration Problems with Automatically Hyperparameter Selection}

\author{
\name{Hang Xie\textsuperscript{a}, Xuewen Li\textsuperscript{c}, Peili Li\textsuperscript{a,b}\thanks{CONTACT Peili Li. Email: lipeili@henu.edu.cn}, and Qiuyu Wang\textsuperscript{d}}
\affil{\textsuperscript{a}School of Mathematics and Statistics, Henan University, Kaifeng 475000, P.R. China. \\ \textsuperscript{b}Center for Applied Mathematics of Henan Province, Henan University, Zhengzhou 450046, P.R. China.\\
\textsuperscript{c}School of Nursing and Health, Henan University, Kaifeng 475000, P.R. China.\\
\textsuperscript{d}School of Software, Henan University, Kaifeng 475000, P.R. China.}
}

\maketitle

\begin{abstract}
In optimization-based image restoration models, the correct selection of hyperparameters is crucial for achieving superior performance. However, current research typically involves manual tuning of these hyperparameters, which is highly time-consuming and often lacks accuracy.
In this paper, we concentrate on the automated selection of hyperparameters in the context of image restoration and present a bilevel programming approach that can simultaneously select the optimal hyperparameters and achieve high-quality restoration results.
For implementation, we reformulate the bilevel programming problem that incorporates an inequality constraint related to the difference-of-convex functions. Following this, we address a sequence of nonsmooth convex programming problems by employing a feasibility penalty function along with a proximal point term.
In this context, the nonsmooth convex programming problem uses the solution of the lower-level problem, which is derived through the alternating direction method of multipliers.
Theoretically, we prove that the sequence generated by the algorithm converges to a Karush-Kuhn-Tucker stationary point of the inequality-constrained equivalent bilevel programming problem.
We conduct a series of tests on both simulations and real images, which demonstrate that the proposed algorithm achieve superior restoration quality while requiring less computing time compared to other hyperparameter selection methods.
\end{abstract}


%%%%%%%%%%%%%%%%%%%%%%%%%%%%%%%%%%%%%%%%%%%%%%%%%%%%%%
\setcounter{equation}{0}
\section{Introduction}\label{section1}
%%%%%%%%%%%%%%%%%%%%%%%%%%%%%%%%%%%%%%%%%%%%%%%%%%%%%%
Restoring a high-quality image from degraded versions, such as those affected by noise, blur, or down-sampling, has various applications in fields like medical imaging, remote sensing,  surveillance, etc. The degradation is often due to a linear blur and an additive random noise. The degraded image can be described mathematically as:
\begin{equation}\label{orig}
	b=\Phi x+\epsilon,
\end{equation}
where $x\in\mathbb{R}^n$ denotes an original vectorized  image, $b\in\mathbb{R}^m$ is a degraded data, $\Phi:\mathbb{R}^n\to\mathbb{R}^m$ is a linear operator which maps the image  to a lower space with $m\ll n$,  and $\epsilon\in\mathbb{R}^m$ is an additive random  noise.
For instance, in the context of magnetic resonance imaging restoration, the linear operator $\Phi$ is defined as $\Phi := \mathcal{P} \mathcal{F}$, where $\mathcal{F}$ denotes the Fourier operator that maps the image space to $k$-space, and $\mathcal{P}$ represents an undersampling mask applied in the $k$-space.
The goal is to reconstruct $x$ from the given degraded data $b$.  However, this data acquisition process
is quite time consuming due to physiological and hardware constraints.

To cope with the ill-posed nature of image restoration, image prior knowledge is usually employed for penalizing the solution to the following minimization problem
$$
\min_{x\in\mathbb{R}^n} \ \mathbb{L}(\Phi x-b)+\lambda\mathbb{P}(x),
$$
where $\mathbb{L}(\cdot)$ is a data-fidelity term and $\mathbb{P}(\cdot)$ is a regularization term denoting image prior, and $\lambda$ is a key hyperparameter. Since prior knowledge about images plays a vital role in the effectiveness of image restoration algorithms, developing effective regularization terms that encapsulate this prior information is fundamental to successful image restoration.
One of the earlier work of Lustig et al. \cite{lustig2007sparse} modeled the regularization term as a linear combination of wavelet sparsity and total variation   regularization \cite{rudin1992nonlinear}.
More specifically, let $ \Psi \in \mathbb{R}^{d \times n} $ represent a Haar wavelet transform matrix. The total variation  and wavelet components are formulated as follows:
\begin{equation}\label{mripro}
	\min_{x \in \mathbb{R}^n} \left\{\frac{1}{2} \| \Phi x - b \|_2^2 + \lambda_1 \|\Psi x\|_1 + \lambda_2 \|x\|_{\text{TV}}\right\},
\end{equation}
where $\|\cdot\|_{\text{TV}}$ denotes a discretized version of total variation regularization, and $\lambda_1, \lambda_2 > 0$ are named hyperparameters.
It is essential to recognize that numerous optimization methods have been developed in the past several years and have   achieved great success in various image restoration applications. For example, the operator splitting method proposed by Ma et al. \cite{ma2008efficient}, the alternating direction method of multipliers (ADMM) employed by Yang et al. \cite{yang2010fast} and Ding et al. \cite{ding2023efficient}, the two-step fixed-point proximity algorithm presented by Li et al. \cite{li2017two}.
The algorithms reviewed are based on the assumption that the hyperparameters $\lambda_1$ and $\lambda_2$ are fixed; nevertheless, selecting suitable values poses substantial challenges in practical applications.

When selecting suitable hyperparameters, the traditional methods such as grid-search and random-search are poorly suited for sparse optimization problems due to their undirected approach. As a result, the limitations of these conventional approaches highlight the necessity of developing a bilevel programming framework, as demonstrated in \cite{kunisch2013bilevel}, which effectively captures the relationships between hyperparameters $\lambda_1,\lambda_2$ and other relevant variable $x$.
Within this framework, the degraded data $b$ (resp.  $\Phi$) are randomly partitioned into three distinct sets: the training set, validation set, and test set, denoted as $b_{\text{tr}}$ (resp. $\Phi_{\text{tr}}$), $b_{\text{val}}$ (resp. $\Phi_{\text{val}}$), and $b_{\text{ts}}$ (resp. $\Phi_{\text{ts}}$), respectively.
As a result, the bilevel programming formulation for selecting hyperparameters is expressed as follows:
\begin{equation}\label{blp}
	\begin{array}{rl}
		\min\limits_{x\in\mathbb{R}^n, \lambda\in\mathbb{R}^2_+} \ &\mathbb{F}(\lambda):=\dfrac{1}{2}\|\Phi_{\text{val}} \ x(\lambda)-b_{\text{val}}\|^{2}_{2} \\
		 \text{s.t.} \ & x(\lambda)\in \argmin\limits_{x \in\mathbb{R}^{n}}
		\Big\{g(x):=\dfrac{1}{2}\|\Phi_{\text{tr}} x-b_{\text{tr}}\|^{2}_{2}+\lambda_{1}\|\Psi x\|_{1}+\lambda_{2}\|x\|_{\text{TV}}\Big\},
	\end{array}
\end{equation}
where $\lambda=(\lambda_1,\lambda_2)$ and $\mathbb{F}(\lambda)$ is called upper objective  function. In this context, the lower-level problem concentrates on recovering the image based on the training data by utilizing fixed hyperparameters $\lambda_1$ and $\lambda_2$. Conversely, the upper-level problem seeks to assess the quality of the restoration image $x(\lambda)$ based on the validation data.
However, the bi-level programming problem \eqref{blp} is generally challenging to solve due to its hierarchical structure and the nested relationships inherent between the upper-level and lower-level problems.

Over the past few decades, numerous strategies have been developed to address bilevel optimization problems, including direct methods and iterative techniques, as detailed in Dempe's book \cite{dempe2020bilevel} and the references therein.
Due to the convexity of the lower-level problem, the bilevel optimization problem \eqref{blp} can be transformed into a single-level problem using its Karush-Kuhn-Tucker (KKT) conditions. However, the resulting single-level problem includes complementarity constraints, complicating the development of a convergence algorithm due to the violation of certain constraint qualifications.
In contrast, the gradient-based iterative methods alternate between solving the upper-level and lower-level problems until convergence is achieved.
For example, Franceschi et al. \cite{franceschi2018bilevel} proposed an iterative gradient descent approach with using hypergradient $\nabla\mathbb{F}(\lambda)$ and established convergence when the lower-level problem is smooth has a unique solution.
This method was improved by Shaban et al. \cite{shaban2019truncated}, who used an approximated solution $x(\lambda)$ to the lower-level problem, and by Liu et al. \cite{liu2018darts} and Ji et al. \cite{ji2021bilevel}, who applied an approximate hypergradient $\nabla\mathbb{F}(\lambda)$, and by Raghu et al. \cite{raghu2019rapid} and Ji et al. \cite{ji2020convergence}, who used a small subset of parameters in the lower-level loop.
In the context of nonsmooth lower-level functions $g(x)$, Feng and Simon \cite{feng2018gradient} examined a gradient descent algorithm that incorporates various commonly utilized penalty functions. Moreover, Bertrand et al. \cite{bertrand2022implicit} and Bertrand et al. \cite{bertrand2020implicit} investigated particular regularized lower-level objective functions, developing  an automatic differentiation method and an efficient implicit differentiation approach for computing the hypergradient $\nabla\mathbb{F}(\lambda)$. Additionally, Grazzi et al. \cite{grazzi2020iteration} established the iteration complexity for several strategies to compute the hypergradient, under the assumption that the lower-level problem is strongly convex.

It is crucial to point out that all the aforementioned gradient descent approaches generally rely on the assumption that the lower-level problem is strongly convex, which enables the efficient computation of the hypergradient $\nabla\mathbb{F}(\lambda)$ using various strategies.
This assumption is fully addressed with the value function-based difference-of-convex algorithm developed by Gao et al. \cite{gao2022value}.
In this method, the lower-level problem is reformulated using inequality constraints that involve difference-of-convex functions, and the resulting  programming problem that coincides with the original bilevel programming problem with respect to their global and local minima.
Subsequently, the main idea of this method \cite{gao2022value} was subsequently improved by Ye et al. \cite{ye2023difference}   to address the difference-of-convex (DC) bilevel programming problem. In this context, the upper-level objective function is formulated as $ \mathbb{F}(\lambda) = \mathbb{F}_1(\lambda) - \mathbb{F}_2(\lambda) $, where  $ \mathbb{F}_1(\lambda) $ and $ \mathbb{F}_2(\lambda) $ are convex functions.
In this paper, we employ the main idea of Gao et al. \cite{gao2022value} to deal with the lower-level problem in the bilevel programming problem \eqref{blp} to get a programming problem contains DC function constraints.
At the current point $(x^k, r^k)$, an efficient alternating direction method of multipliers (ADMM) is employed to solve the lower-level problem associated with the concave part of the constraints. This process yields a pause point $\bar x^{k}$ and a corresponding multiplier $\xi^k$.
Following this, we linearize the concave portion of the DC functions structure and employ a feasibility penalty function along with a proximal point term to create a nonsmooth convex single-level programming problem.
Finally, we solve the resulting convex programming problem to obtain the next point $(x^{k+1}, r^{k+1})$ and terminate the iterative process once a specified convergence criterion is met.
Theoretically, we prove that the sequence $\{(x^k,r^k)\}$ generated by the proposed algorithm convergences globally to the KKT point of the resulting equivalent problem.
We perform numerical experiments using real image data and find that our proposed algorithm for \eqref{mripro} achieves better quality in recovered images with less computing time compared to other popular methods.

The remaining parts of this paper are organized as follows.  In section \ref{sec2}, we construct the DC functions based penalty algorithm to solve the bilevel programming problem \eqref{blp} for selecting hyperparameters and restoring images.
In section \ref{sec3}, we provide theoretical analysis and show that the proposed algorithm converges to the KKT point of the equivalent inequality constrained optimization problem.
In section \ref{sec4}, we test the presented algorithm and do performance comparisons on a real data set by several numerical experiments.  Finally, we conclude our paper with some remarks in section \ref{sec5}.

%%%%%%%%%%%%%%%%%%%%%%%%%%%%%%%%%%%%%%%%%%%%%%%%%%%%%%%%%%%%%%%%%
\section{Algorithm's construction}\label{sec2}
\setcounter{equation}{0}
%%%%%%%%%%%%%%%%%%%%%%%%%%%%%%%%%%%%%%%%%%%%%%%%%%%%%%%%%%%%%%%%
In this section, we focus on developing an efficient DC functions based penalty algorithm to solve the bilevel programming problem \eqref{blp}.
At the first place, it has been shown in this literature that, solving the total variation and wavelet regularized image restoration problem \eqref{mripro} is equivalent to solving the following inequality constrained programming problem
\begin{equation}\label{mripro2}
	\min_{x \in \mathbb{R}^n} \left\{\frac{1}{2} \| \Phi x - b \|_2^2 \quad \text{s.t.} \ \|\Psi x\|_1 \leq r_1, \ \|x\|_{\text{TV}}\leq r_2\right\},
\end{equation}
where $r=(r_1,r_2)$ is a suitable positive vector which named hyperparameter here. However, just as selecting $\lambda\in\mathbb{R}^2_+$ in \eqref{mripro} poses challenges, finding an appropriate $r\in\mathbb{R}^2_+$ is also a difficult task.
Hence, similar to the bilevel programming problem \eqref{blp}, we  also jointly select the hyperparameter and restore the image, reformulating it as the following inequality-constrained bilevel programming problem
\begin{equation}\label{blp2}
	\begin{array}{rl}
		\min\limits_{x\in\mathbb{R}^n, r\in\mathbb{R}^2_+} \ & \mathbb{F}(r):=\dfrac{1}{2}\|\Phi_{\text{val}} \ x(r)-b_{\text{val}}\|^{2}_{2}   \\
		\text{s.t.} \                                        & x(r)\in \argmin\limits_{x \in\mathbb{R}^{n}}
		\Big\{\dfrac{1}{2}\|\Phi_{\text{tr}} x-b_{\text{tr}}\|^{2}_{2} \quad \text{s.t.} \ \|\Psi x\|_1 \leq r_1, \ \|x\|_{\text{TV}}\leq r_2\Big\}.
	\end{array}
\end{equation}
In contrast to \eqref{blp}, we see that the lower-level problem of \eqref{blp2} aims to find an optimal solution $x(r)$ for fixed values of $r_1$ and $r_2$, while the upper-level objective function assesses the quality of this solution.
Moreover, the lower-level problem admits a unique solution and that some of the gradient descent approaches reviewed before seemly can be employed. Nevertheless, addressing \eqref{blp2} may pose additional difficulties, especially when calculating the hypergradient $\nabla\mathbb{F}(r)$, owing to the existence of inequality constraints in the lower-level problem.

In order to solve \eqref{blp2}, we use the idea of Gao et al. \cite{gao2022value} that reformulate the lower-level problem in \eqref{blp2} as an inequality constraint contains an implicit minimization problem in the form of
\begin{equation}\label{blp3}
	\begin{array}{rl}
		\min\limits_{x\in\mathbb{R}^n, r\in\mathbb{R}^2_+} \ & \mathbb{F}(r):=\dfrac{1}{2}\|\Phi_{\text{val}} \ x-b_{\text{val}}\|^{2}_{2}           \\
		\text{s.t.} \                                        &
		 \dfrac{1}{2}\|\Phi_{\text{tr}} x-b_{\text{tr}}\|^{2}_{2}-h(r)\leq 0, \ \ \|\Psi x\|_1 \leq r_1, \ \ \|x\|_{\text{TV}}\leq r_2,
	\end{array}
\end{equation}
where
\begin{equation}\label{3}
	h(r) \triangleq \min_{x\in\mathbb{R}^n}\Big\{\dfrac{1}{2}\|\Phi_{\text{tr}} x-b_{\text{tr}}\|^{2}_{2} \quad \text{s.t.} \ \|\Psi x\|_1 \leq r_1, \ \|x\|_{\text{TV}}\leq r_2\Big\}.
\end{equation}
From the formulation, we can observe that the process initially involves solving the lower-level problem \eqref{3} with fixed values of $ r_1$ and $r_2$ using the training data to obtain an optimal solution $x(r) $. Subsequently, the quality of the solution  is evaluated using the validation set, which is reflected in the objective function of the upper-level problem while the constraints are inherently satisfied.

In order to solve the resulting bilevel programming problem \eqref{blp3}, we focus on an iterative approach to solve the lower-level problem \eqref{3} using current iterate $x^k$ and current values of $(r_1^k,r_2^k)$. For this purpose, we let $u:=\Phi_{\text{tr}}x-b_{\text{br}}\in\mathbb{R}^m$ and $v:=\Psi x\in\mathbb{R}^d$, then \eqref{3} can be rewritten equivalently as
\begin{equation}\label{3-1}
	h(r^k) \triangleq \min_{x\in\mathbb{R}^n,u\in\mathbb{R}^m,v\in\mathbb{R}^d}\bigg\{\dfrac{1}{2}\|u\|^{2}_{2} +\delta_{\mathbb{B}_1^{(r_1^k)}}(v) +\delta_{\mathbb{B}_{\text{TV}}^{(r_2^k)}}(x)\quad \text{s.t.} \ \Phi_{\text{tr}} x-u=b_{\text{tr}}, \ \Psi x-v=0 \bigg\},
\end{equation}
where the symbol $\delta_{\mathbb{B}}(\cdot)$ is an indicator function defined on a closed convex set $\mathbb{B}$, and $\mathbb{B}_p^{(q)}$ denotes an $\ell_p$-norm ball with radius $q$. The nonsmooth convex problem $ \eqref{3-1} $ can be addressed efficiently using the symmetric Gauss-Seidel based ADMM with $(u,v)$ is one group and $x$ is another. For further details, one may refer to Li et al. \cite{2019A}.
Let $\bar x^k$ be the optimal solution of \eqref{3-1} and   $\xi^k\in\mathbb{R}^2_+$ be the multiplier corresponding to the constraints in \eqref{3}. It is easy to show that $(\bar x^k,\xi^k)\in\mathcal{M}(r^k)$ with $\M(r)$ taking the following form
\begin{equation}\label{6}
	\begin{array}{lll}		
		\M (r)\triangleq \bigg\{(x, \xi) \in \mathbb{R}^n\times\mathbb{R}^{2}_{+} \ | \ 0 \in  \Phi_{\text{tr}}^\top(\Phi_{\text{tr}} x- b_{\text{tr}})+\xi_{1}\partial \|\Psi x \|_{1}+ \xi_{2} \sum^{n}\limits_{i=1}\partial \|D_{i}x\|_{1},\\
		\hspace{12.8em}  \xi_{1}\big( \|\Psi x \|_{1} - r_{1}\big)=0, \  \xi_{2}\big(\sum^{n}\limits_{i=1}\|D_{i}x\|_{1} - r_{2}\big)=0\bigg\},
	\end{array}
\end{equation}
where the notation $\sum^{n}_{i=1}\|D_{i}x\|_{1}$ is denoted as the total variation $\|x\|_{\text{tv}}$ with a difference operator $D_i$ for $i=1,\ldots,n$.

While the solution $\bar x^k$ to the lower-level problem \eqref{3} and the corresponding multiplier $\xi^k$ is computed, we turn to consider the following minimization problem based on feasibility penalty function
\begin{equation}\label{66}
	\begin{array}{rl}
		\min\limits_{x\in\mathbb{R}^n,r\in\mathbb{R}^{2}_{+}} \ \phi_{k}(x,r)\triangleq& \frac12\|\Phi_{\text{val}} x-b_{\text{val}}\|_2^2+\frac{\rho_k}{2}\|x-x^{k}\|_2^{2} +\frac{\rho_k}{2}\|r-r^{k}\|_2^{2}\\
		   & +\alpha_{k} \max \Big\{0, \ \theta_{k}(x,r), \ \|\Psi x\|_{1}- r_{1}, \ \sum^{n}\limits_{i=1}\|D_{i}x\|_{1}- r_{2} \Big\},
	\end{array}
\end{equation}
were  $\rho_k>0$ is a proximal parameter and $\alpha_k>0$ is a penalty parameter.
In practice, $\rho_k$ and $\alpha_k$ are adaptively changed to enhance numerical stability and constraints enforcement.
  In the context, $\theta_k(x,r)$ is defined as follows:
\begin{equation}\label{7}
	\theta_{k}{(x,r) \triangleq \dfrac{1}{2}\|\Phi_{\text{tr}} x-b_{\text{tr}}\|^{2}_{2} -\dfrac{1}{2}\|\Phi_{\text{tr}} \bar x^k-b_{\text{tr}}\|^{2}_{2}+\langle\xi^{k},r-r^{k}\rangle}.
\end{equation}
It is important to note that terms $\|x-x^k\|_2^2$ and $\|r-r^k\|_2^2$ in \eqref{6} are named proximal terms which are used to control the next iteration point is not far away from $(x^k,r^k)$.
Additionally, it is noted that $\theta_k(x,r)$, as defined in \eqref{7}, serves as an approximation expression for the first inequality constraint in \eqref{blp3}. This is because the term $h(r)$ has been linearized at the point $\bar x^k$.
Finally, it is also to note that
$$
\dfrac{1}{2}\|\Phi_{\text{tr}} \bar x^k-b_{\text{tr}}\|^{2}_{2}=h(r^k)
$$
which is used in the subsequent theoretical analysis.

To facilitate our discussion, we introduce the notation $z:=(x,r)\in\mathbb{R}^{n+1}$ and define $\Sigma :=\mathbb{R}^{n}\times \mathbb{R}^{2}_{+}$. Utilizing this notation, we observe that solving the problem as posed in \eqref{6} is equivalent to finding a point $z^{k+1}$ such that $0\in\partial\phi_k(z^{k+1})+\mathbb{N}_\Sigma(z^{k+1})$. For practical implementation, it is reasonable to seek an inexact solution $z^{k+1}$ that satisfies the following inequality:
\begin{equation}\label{9}
	\dist\Big(0,\partial \phi_{k}(z^{k+1})+\mathbb{N}_{\Sigma}(z^{k+1})\Big)\leq \frac{\sqrt{2}}{2}\rho_k \|z^{k}-z^{k-1}\|,
\end{equation}
where $\rho_k\geq 0$ is a positive scalar. In this context, $\dist(x,\Omega)$ denotes a distance from $x$ to a closed convex set $\Omega$, and
$\mathbb{N}_{\Sigma}(z)$ represents a normal cone on $\Sigma$ at $z\in\Sigma$.
Furthermore, \eqref{6} is a nonsmooth convex optimization problem and can be efficiently solved using existing subgradient methods, such as Nesterov's optimal method for max-type functions \cite{nesterov2014subgradient}. For brevity, we do not elaborate on the numerical solution of this problem, as it is beyond the scope of this paper. Here, we emphasize that problem \eqref{6} is solved iteratively, yielding a sequence $\{z^k\}$. The iterative process is terminated once the stopping criterion
$
\max\big\{\|z^{k+1}-z^k\|_2,\eta^{k+1}\big\}\leq \text{tol}
$
is satisfied, where $\text{tol}$ denotes a prescribed tolerance, and the quantity $\eta^{k+1}$ is defined as
\begin{equation}\label{etak}
\eta^{k+1}\triangleq \max\Big\{0, \ \theta_{k}(x^{k+1},r^{k+1}), \  \|\Psi^{*}x^{k+1}\|_{1}- r^{k+1}_{1}, \  \sum_{i=1}^{n}\|D_{i}x^{k+1}\|_{1}- r^{k+1}_{2} \Big\}.
\end{equation}
Taking all the above into account, we summarize the complete procedure of the proposed DC-based algorithm for the bilevel framework, referred to as BF-DCA, as follows:
\begin{algorithm}[H]
	\caption{BF-DCA}
	\label{alg:BF-DCA}
	\begin{algorithmic}[1]
		\Require Initial point $(x^{0}, r^{0}) \in \mathbb{R}^n \times \mathbb{R}^2_{+}$;
		parameters $c_{\rho}>0$, $c_{\alpha} > 0$, $\delta_{\alpha} > 0$, $\delta_{\rho}>0$, and $\tau>0$;
		initial parameters $\alpha_{\max}>\alpha_{0} > 0$ and $\rho_{\max}>\rho_0>0$;  small
		tolerance $\text{tol} > 0$.
		\Ensure An approximate solution $(x^{k}, r^{k})$.
		
		\For{$k = 0,1,2,\ldots$}
		\State Solve \eqref{3} by ADMM to obtain a solution $\bar{x}^{k}$ and the corresponding  multiplier $\xi^{k}$;
		
		\State Solve \eqref{6}  by a suitable subgradient method up to the tolerance specified in \eqref{9}, and obtain $(x^{k+1},r^{k+1})$;
		\State Set $z^{k+1} \gets (x^{k+1}, r^{k+1})$;
		
		\If{$\max\{\|z^{k+1}-z^{k}\|_2,\, \eta^{k+1}\} < \text{tol}$ with $\eta^{k+1}$ definded in \eqref{etak}}
		\State \textbf{break};
		\EndIf
		
		\State Set $\Delta^{k+1} \gets \|z^{k+1}-z^{k}\|_2$;
		
		\If{$\max\left\{\alpha_k,{1}/{\eta^{k+1}}\right\} < {c_{\alpha}}/{\Delta^{k+1}}$}
		\State $\alpha_{k+1} \gets \min\{\alpha_k + \delta_{\alpha},\alpha_{\max}\}$;
		\Else
		\State $\alpha_{k+1} \gets \alpha_k$;
		\EndIf
		\If{$ \Delta^{k+1}>c_{\rho}$}
		\State $\rho_{k+1} \gets \min\{\rho_k+\delta_\rho,\rho_{\max}\}$;
		\Else
		\State $\rho_{k+1} \gets \rho_k$;
		\EndIf
		\EndFor
	\end{algorithmic}
\end{algorithm}

%%%%%%%%%%%%%%%%%%%%%%%%%%%%%%%%%%%%%%%%%%%%%%%%%%%%%%%%%%%%%%%%%
\section{Theoretical analysis}\label{sec3}
\setcounter{equation}{0}
%%%%%%%%%%%%%%%%%%%%%%%%%%%%%%%%%%%%%%%%%%%%%%%%%%%%%%%%%%%%%%%%
In this section, we demonstrate the convergence results of the BF-DCA algorithm under certain specified conditions. To achieve this, we first establish the equivalence between problem \eqref{blp} and problem \eqref{blp2}. Following this, we show that the sequence $\{(x^k, r^k)\}$, generated by the BF-DCA algorithm, converges to a KKT point of problem \eqref{blp3}. Consequently, the convergence result of BF-DCA can be directly inferred from the established equivalence between problems \eqref{blp2} and \eqref{blp3}, as demonstrated by Gao et al. \cite{gao2022value}.

\subsection{Relationship between problem \eqref{blp} and problem  \eqref{blp2}}

At the beginning, we give the the definition of KKT point to problem  \eqref{blp3}.

\begin{definition}[KKT point]\label{kktp}
		A feasible point $(\bar x,\bar r)$ is called a KKT point
		of problem~\eqref{blp3} if there exist Lagrange multipliers
		$\xi_0 \ge 0$, $\xi_1 \ge 0$, and $\xi_2 \ge 0$
		such that the following conditions hold:
\begin{equation}\label{eq:KKT_compact}
	\left\{
	\begin{aligned}
		0 &\in \nabla_x \mathcal L(\bar x,\bar r,\xi_0,\xi_1,\xi_2), \\[0.6ex]
		0 &\in \nabla_r \mathcal L(\bar x,\bar r,\xi_0,\xi_1,\xi_2)
		+ \mathbb N_{\mathbb R^2_+}(\bar r), \\[0.6ex]
		0 &\le (\xi_0,\xi_1,\xi_2)\ \perp\
		\begin{pmatrix}
			\dfrac12\|\Phi_{\mathrm{tr}}\bar x-b_{\mathrm{tr}}\|_2^2-h(\bar r) \\[0.8ex]
			\|\Psi\bar x\|_1-\bar r_1 \\[0.6ex]
			\|\bar x\|_{\mathrm{TV}}-\bar r_2
		\end{pmatrix}
		\le 0 ,
	\end{aligned}
	\right.
\end{equation}
where
\begin{equation}\label{eq:Lagrangian}
	\begin{aligned}
		\mathcal L(x,r,\xi_0,\xi_1,\xi_2)
		= \ \frac12\|\Phi_{\mathrm{val}}x-b_{\mathrm{val}}\|_2^2
		+ \xi_0\!\left(\frac12\|\Phi_{\mathrm{tr}}x-b_{\mathrm{tr}}\|_2^2-h(r)\right)
		+ \xi_1(\|\Psi x\|_1-r_1)
		+ \xi_2(\|x\|_{\mathrm{TV}}-r_2).
	\end{aligned}
\end{equation}
\end{definition}

We now show that problem~\eqref{mripro} with hyperparameters
$\lambda_1$ and $\lambda_2$ is equivalent to problem \eqref{mripro2}
with hyperparameters $r_1$ and $r_2$, provided that these parameters
are chosen appropriately.
Let $\mathbb{S}_p(\lambda)$ denote the solution set of problem \eqref{mripro},
and let $\mathbb{S}_c(r)$ denote the solution set of problem \eqref{mripro2}. We now show that if $x \in \mathbb{S}_p(\lambda)$, then $x \in \mathbb{S}_c(r)$,
and vice versa.
This result has been known in the literature for many years, and its proof is straightforward.
For completeness, we provide the proof here.

\begin{lemma}\label{lem1}
	Suppose that $\bar x \in \mathbb{S}_p(\lambda)$ with $\lambda \in \mathbb{R}^2_+$,
	and let $r = (r_1,r_2) = (\|\Psi \bar x\|_1, \|\bar x\|_{\mathrm{TV}}) \in \mathbb{R}^2_+$.
	Then, we have $\bar x \in \mathbb{S}_c(r)$.
\end{lemma}
\begin{proof}
Note that $\bar x$ is an optimal solution of problem~\eqref{mripro}.
Then, it satisfies
\begin{equation}\label{opt_condition}
	0 \in \Phi^\top (\Phi \bar x - b)
	+ \lambda_1 \, \Psi^\top \partial \|\Psi \bar x\|_1
	+ \lambda_2 \, \partial \|\bar x\|_{\mathrm{TV}}.
\end{equation}
On the other hand, the KKT system of problem~\eqref{mripro2} can be written as
\begin{equation}\label{KKT_mripro2}
	\left\{
	\begin{aligned}
		&0 \in \Phi^\top (\Phi x - b)
		+ \xi_1 \, \Psi^\top \partial \|\Psi x\|_1
		+ \xi_2 \, \partial \|x\|_{\mathrm{TV}},\\[0.5em]
		&\|\Psi x\|_1 \le r_1, \quad \|x\|_{\mathrm{TV}} \le r_2,\\[0.3em]
		&\xi_1 \ge 0, \quad \xi_2 \ge 0,\\[0.3em]
		&\xi_1 (\|\Psi x\|_1 - r_1) = 0, \quad
		\xi_2 (\|x\|_{\mathrm{TV}} - r_2) = 0.
	\end{aligned}
	\right.
\end{equation}
Hence, noting that
$
r = (r_1,r_2) = (\|\Psi \bar x\|_1, \|\bar x\|_{\mathrm{TV}}),
$
we see that $\bar x$ also satisfies the KKT system~\eqref{KKT_mripro2},
which implies that $\bar x \in \mathbb{S}_c(r)$.
\end{proof}

\begin{lemma}\label{lem2}
Suppose that $\bar x \in \mathbb{S}_c(r)$ with $r \in \mathbb{R}^2_+$,
and let $\lambda = (\lambda_1,\lambda_2)$ be the Lagrange multipliers of problem~\eqref{mripro2} at $\bar x$.
Then, it follows that $\bar x \in \mathbb{S}_p(\lambda)$.
\end{lemma}
\begin{proof}
	Note that $\bar x \in \mathbb{S}_c(r)$.
	Then, the KKT system~\eqref{KKT_mripro2} is satisfied at $\bar x$.
	By setting $\lambda_1 = \xi_1$ and $\lambda_2 = \xi_2$,
	the optimality condition~\eqref{opt_condition} is also satisfied,
	which implies that $\bar x \in \mathbb{S}_p(\lambda)$.
\end{proof}

The coincidence of the solution sets of problems~\eqref{mripro} and~\eqref{mripro2} implies the coincidence of the solution sets of the lower-level problems in \eqref{blp} and \eqref{blp2}.
However, this alone is not sufficient to guarantee that the global solutions of both problems coincide.
We now give a lemma which is important to following developments.

\begin{lemma}\label{lem3}
	Assume that $\mathcal{M}(\bar r)=\{(\bar x,\bar\xi)\}$ is a singleton.
	Let $(x,r)$ be a feasible solution of problem~\eqref{blp2}, and let
	$\mathbb{D}\subset \mathbb{R}^n\times\mathbb{R}^2_+$ be a dense subset.
	Suppose that there exist a bounded set $\Lambda$ and a sufficiently small scalar
	$\varepsilon>0$ such that, for any
	$(x,r)\in \mathbb{D}\cap \mathbb{B}_{\varepsilon}(\bar x,\bar r)$, one has
	$
	\Lambda \cap \mathcal{M}(r)\neq \varnothing.
	$
	Then there exists a sufficiently small constant $\varepsilon_1>0$ with
	$0<\varepsilon_1<\varepsilon$ such that
	$$
	\Lambda \cap \mathcal{M}(r)\subseteq
	\mathbb{B}_{\varepsilon_1}(\bar x,\bar\xi)
	$$
	for all $(x,r)\in \mathbb{B}_{\varepsilon}(\bar x,\bar r)$.
\end{lemma}
\begin{proof}
	Suppose that $\varepsilon_1>0$ is chosen sufficiently small such that
	$\mathbb{B}_{\varepsilon_1}(\bar x,\bar r)
	\subset \mathbb{D}\cap \mathbb{B}_{\varepsilon}(\bar x,\bar r)$.
	Then there exists a sequence
	$\{(x^k,r^k)\}\subset \mathbb{B}_{\varepsilon_1}(\bar x,\bar r)$ such that
	$(x^k,r^k)\to(\bar x, \bar r)\quad \text{as } k\to\infty$,
	and the corresponding multipliers $\xi^k$ satisfy
	$(\bar x^k,\xi^k)\in \Lambda\cap \mathcal{M}(r^k)$.
	Assume, for the sake of contradiction, that
	$(\bar x^k,\xi^k)\notin \mathbb{B}_{\varepsilon_1}(\bar x,\bar\xi)$,
	which implies that
	$\|\xi^k-\bar\xi\|>\varepsilon_1$ for all  $k$.
		
	Note that $\Lambda$ is a bounded set and that $(\bar x^k,\xi^k)\in \Lambda$.
	It follows that the sequence $\{\xi^k\}$ is also bounded, so there exists a subsequence
	$\{(\bar x^{k_j}, \xi^{k_j})\}$ converging to some $(\bar x^*,\xi^*)\in \Lambda$.
	Since $\|\xi^k-\bar\xi\|>\varepsilon_1$ for all $k$, we have $\|\xi^{k_j}-\bar\xi\|>\varepsilon_1$ and thus
	$\|\xi^*-\bar\xi\|> \varepsilon_1$.	
	On the other hand, because $(\bar x^k,\xi^k)\in \Lambda\cap \mathcal{M}(r^k)$
	and $(x^k,r^k)\to(\bar x,\bar r)$, it follows that
	$(\bar x^{k_j},\xi^{k_j})\in \Lambda\cap \mathcal{M}(r^{k_j})$
	and $(\bar x^{k_j},r^{k_j})\to(\bar x^*,\bar r)$.
	Since $\xi^{k_j}\to\xi^*$, we deduce that
	$(\bar x^*,\xi^*)\in \mathcal{M}(\bar r)$.
	Recalling that $\mathcal{M}(\bar r)=\{(\bar x,\bar\xi)\}$ is a singleton,
	we must have $\bar x^*=\bar x$ and $\xi^*=\bar\xi$, which contradicts
	$\|\xi^k-\bar\xi\|> \varepsilon_1$ for all $k$.	
	Hence, the assumption is wrong, which completes the proof.
\end{proof}

We now prove the coincidence of the global solution to problems   \eqref{blp} and \eqref{blp2}. For convenience, we denote
$$
\mathbb{Q}(x)\triangleq \Big(\|\Psi x\|_1 , \|x\|_{\text{TV}}\Big) \in\mathbb{R}^2_+.
$$

\begin{theorem}\label{globth}
	Let $(\bar x,\bar r)\in\mathbb{R}^{n}\times\mathbb{R}^2_+$ be a global optimal solution of the bilevel problem \eqref{blp2} and $(\bar x,\bar\xi)\in\mathcal{M}(\bar r)$, then $(\bar x,\bar \lambda)$ is a also a global optimal solution of \eqref{blp} with $\bar\lambda=\bar\xi$.
	On the contrary, suppose that there there exists a dense subset $\mathbb{D}\subset\mathbb{R}^n\times\mathbb{R}^2_+$ such that $\mathcal{M}(r)\neq\varnothing$ for all $(x,r)\in\mathbb{D}$. Let $(\bar x,\bar \lambda)$ be a global optimal solution to problem  \eqref{blp} and let $\bar r=\mathbb{Q}(\bar x)$. Then $(\bar x,\bar r)$ is a also a global optimal solution of problem \eqref{blp2}.
\end{theorem}
\begin{proof}
	Using the notations of $\mathbb{S}_p(\lambda)$ and $\mathbb{S}_c(r)$, we can restate the optimization problems \eqref{blp} and \eqref{blp2} in an equivalent form as follows:
	\begin{equation}\label{blp-1}
			\min\limits_{x\in\mathbb{R}^n, \lambda\in\mathbb{R}^2_+} \ \Big\{ \dfrac{1}{2}\|\Phi_{\text{val}} \ x-b_{\text{val}}\|^{2}_{2} \ | \  x\in  \mathbb{S}_p(\lambda)\Big\},
	\end{equation}
    and
	\begin{equation}\label{blp2-1}
		\min\limits_{x\in\mathbb{R}^n, r\in\mathbb{R}^2_+} \ \Big\{ \dfrac{1}{2}\|\Phi_{\text{val}} \ x-b_{\text{val}}\|^{2}_{2} \ | \  x\in  \mathbb{S}_c(r)\Big\}.
	\end{equation}
	
	 Since $(\bar{x}, \bar{r}) \in \mathbb{R}^n \times \mathbb{R}^2_+$ is a global optimal solution to problem \eqref{blp2-1}, it follows that $\bar{x} \in \mathbb{S}_c(\bar{r})$. Additionally, because $(\bar{x}, \bar{\xi}) \in \mathcal{M}(\bar{r})$, we conclude that $\bar{x} \in \mathbb{S}_p(\bar{\xi})$ from Lemma \ref{lem2}. This implies that $\bar{x} \in \mathbb{S}_p(\bar{\lambda})$ with $\bar{\lambda} := \bar{\xi}$, that is to say $\bar{x}$ is a feasible solution for problem \eqref{blp-1}. Furthermore, suppose that $(x,\lambda)$ with $\lambda\in\mathbb{R}^2_+$ is a feasible point of \eqref{blp-1}, i.e., $x\in\mathbb{S}_p(\lambda)$, then from Lemma \ref{lem1}, it gets that $x\in\mathbb{S}_c(r)$ if  choosing $r:=\mathbb{Q}(x)$, i.e., $(x,r)$ is a feasible solution of \eqref{blp2-1}.
	 Recall that $(\bar x,\bar r)$ is a global optimal solution of \eqref{blp2-1}, it gets that $\|\Phi_{\text{val}} \ \bar x-b_{\text{val}}\|^{2}_{2}\leq \|\Phi_{\text{val}} \ x-b_{\text{val}}\|^{2}_{2}$, which shows $\bar x$ is a global optimal solution of problem \eqref{blp-1}.
	
	Since $(\bar{x}, \bar{\lambda}) \in \mathbb{R}^n \times \mathbb{R}^2_+$ is a global optimal solution to problem \eqref{blp-1}, it follows that $\bar{x} \in \mathbb{S}_p(\bar{\lambda})$. From Lemma \ref{lem1}, we know that $\bar x\in\mathbb{S}_c(\bar r)$ if choosing $\bar r=\mathbb{Q}(\bar x)$, which shows $(\bar x,\bar r)$ is a feasible solution of problem \eqref{blp2-1}. Suppose that $(x,r)\in\mathbb{D}$ is a feasible point of problem \eqref{blp2-1}, then there a sequence $\{(\bar x^k,r^k)\}\subseteq\mathbb{D}$ such that $\{(\bar x^k,r^k)\}\to(x,r)$. Since $\mathcal{M}(r^k)\neq\varnothing$ for  $(\bar x^k,r^k)\in\mathbb{D}$, let $(\bar x^k,\lambda^k)\in\mathcal{M}(r^k)$, i.e., $0\in \Phi_{\text{tr}}^\top (\Phi_{\text{tr}} \bar x^k - b_{\text{tr}})
	+ \lambda_1^k \, \Psi^\top \partial \|\Psi \bar x^k\|_1
	+ \lambda_2^k \, \partial \|\bar x^k\|_{\mathrm{TV}}$, which means $\bar x^k\in\mathbb{S}_p(\lambda^k)$, i.e., $\bar x^k$ is a feasible solution of \eqref{blp-1}.
	Since $(\bar{x}, \bar{\lambda})$ is an optimal solution to problem \eqref{blp-1}, it follows that $\|\Phi_{\text{val}} \ \bar x-b_{\text{val}}\|^{2}_{2}\leq \|\Phi_{\text{val}} \ \bar x^k-b_{\text{val}}\|^{2}_{2}$. Let $k\to\infty$, it gets that $\|\Phi_{\text{val}} \ \bar x-b_{\text{val}}\|^{2}_{2}\leq \|\Phi_{\text{val}} \ x-b_{\text{val}}\|^{2}_{2}$, which shows $\bar x$ is a global optimal solution of problem \eqref{blp2-1}.

\end{proof}

We now prove that the local solution to problems   \eqref{blp} and \eqref{blp2} are also coincidence.

\begin{theorem}\label{localth}
	Let $(\bar x,\bar r)\in\mathbb{R}^{n}\times\mathbb{R}^2_+$ with $\bar r=\mathbb{Q}(\bar x)$ be a local optimal solution of the bilevel problem \eqref{blp2} and $(\bar x,\bar\xi)\in\mathcal{M}(\bar r)$, then $(\bar x,\bar \lambda)$ is a also a local optimal solution of \eqref{blp} with $\bar\lambda:=\bar\xi$.
\end{theorem}	
\begin{proof}
	Since $(\bar x,\bar r)$ is a local optimal solution of   problem \eqref{blp2}, then there exists a small $\varepsilon_0>0$ such that for all $x\in\mathbb{S}_c(r)$ with $r\in\mathbb{R}^2_+$ and $(x,r)\in\mathbb{B}_{\varepsilon_0}(\bar x,\bar r)$, it holds that $\|\Phi_{\text{val}}  \bar x-b_{\text{val}}\|^{2}_{2}\leq \|\Phi_{\text{val}}  x-b_{\text{val}}\|^{2}_{2}$. Furthermore, since $\bar x\in\mathbb{S}_c(\bar r)$ and that $(\bar x,\bar\lambda)\in\mathcal{M}(\bar r)$, we get from Lemma \ref{lem2} that $\bar x\in\mathbb{S}_p(\bar\lambda)$, which shows $\bar x$ is a feasible point of \eqref{blp}.
	Suppose that $(x,\lambda)$ is an arbitrary feasible point of \eqref{blp}, i.e., $x\in\mathbb{S}_p(\lambda)$.
	From Lemma \ref{lem1} we know that $x\in\mathbb{S}_c(r)$ when $r=\mathbb{Q}(x)$. Additionally, since $\mathbb{Q}(x)$ is a continuous function, then there exists a $0<\varepsilon_1<\varepsilon_0$ such that $r=\mathbb{Q}(x)\in\mathbb{B}_{\varepsilon_0}(\bar r)$ when $(x,\lambda)\in\mathbb{B}_{\varepsilon_1}(\bar x,\bar\lambda)$.
	Hence, we get $\|\Phi_{\text{val}}  \bar x-b_{\text{val}}\|^{2}_{2}\leq \|\Phi_{\text{val}}  x-b_{\text{val}}\|^{2}_{2}$, i.e., $(\bar x,\bar\lambda)$ is a local optimal solution of \eqref{blp}.
		
\end{proof}

This lemma provides an explanation for our decision to solve the lower-level problem with inequality constraints, as represented by \eqref{blp2}, instead of addressing the original bilevel programming problem stated in \eqref{blp}.

\subsection{Convergence result of BF-DCA to problem \eqref{blp3}}

In this subsection, we focus on establish the convergence theorem of the algorithm BF-DCA under some appropriate conditions.
Firstly, we note that for the current value of $r^k$, the solution set $\mathbb{S}_c(r^k)$ of the lower-level problem \eqref{blp2} is non-empty, and both the optimal solution $\bar{x}^k$ to problem \eqref{3} and its corresponding multiplier $\xi^k$ exist, indicating that the set $\mathcal{M}(r^k)$ is also non-empty.
Next, we can conclude from Ye et al. \cite[Lemma 3]{ye2023difference} that the function $h(r)$ defined in \eqref{3} is Lipschitz continuous, and, based on Ye et al. \cite[Theorem 3]{ye2023difference}, we can infer that it is subdifferentiable which can be described as follows:

\begin{lemma}\label{lemm33}
	For any $r \in \mathbb{R}^2_+$, we obtain that $\partial h(r) = \{-\xi\}$, where $(x, \xi) \in -\mathcal{M}(r)$, with $x$ being an optimal solution of \eqref{3} and $\xi$ its corresponding multiplier.
\end{lemma}

Based on this lemma and the result established in Ye et al. \cite[Theorem 1]{ye2023difference}, we present the following theorem without proof.

\begin{theorem}\label{them32}
	Let $\{z^k\}$ and $\{\alpha^k\}$ be bounded sequences generated by BF-DCA. If $\bar r \in \mathbb{R}^2_+$, then any cluster point $\bar z := (\bar x, \bar r)$ of the sequence $\{z^k\}$ is a KKT point of the problem described in \eqref{blp2}.
\end{theorem}

From Theorem \ref{them32}, we consider the subsequence $\{(x^{k_j}, r^{k_j})\}_{k_j \in \mathcal{K}}$, where $\mathcal{K} := \{0, 1, 2, \ldots\}$, that converges to a cluster point $(\bar x, \bar r)$, specifically, $\{(x^{k_j}, r^{k_j})\}_{k_j \in \mathcal{K}} \to (\bar x, \bar r)$.
In the following, we will demonstrate that under the Kurdyka-Lojasiewicz property, the convergence of a subsequence can be extended to the entire sequence $\{(x^k, r^k)\}$.
For this purpose, we state the definition of the Kurdyka-Lojasiewicz (KL) property, originally developed in Attouch et al.~\cite{attouch2010proximal} and later formalized in Attouch et al.~\cite{attouch2013convergence}.

\begin{definition} \cite[Kurdyka--Lojasiewicz property]{attouch2013convergence,attouch2010proximal}\label{dfkl}
	Let $f:\mathbb{R}^n\to(-\infty,+\infty]$ be a proper closed convex function.
	We say that $f$ satisfies KL property
	at a point $\bar{x}\in\mathrm{dom}\,\partial f$
	if there exist $\eta\in(0,+\infty]$, a neighborhood $\mathbb{U}$ of $\bar{x}$,
	and a function $\varphi:[0,\eta)\to\mathbb{R}_{+}^1$ such that
   (i) $\varphi(0)=0$; (ii) $\varphi$ is continuous on $[0,\eta)$ and continuously differentiable on $(0,\eta)$; (iii) $\varphi$ is concave and $\varphi'(s)>0$ for all $s\in(0,\eta)$;
	and for all $x\in\mathbb{ U}$ satisfying
	$	f(\bar{x}) < f(x) < f(\bar{x})+\eta	$,
	the following inequality holds:
	$$
	\varphi'\Big(f(x)-f(\bar{x})\Big)\,
	\mathrm{dist}\Big(0,\partial f(x)\Big)\ge 1.
	$$
\end{definition}
It is well known that the $\ell_1$-norm and the total variation regularization functions appearing in \eqref{blp2} are semi-algebraic and hence satisfy the KL property.

\begin{definition}\cite[Uniform Kurdyka--Lojasiewicz property]{attouch2013convergence}
	Let $f:\mathbb{R}^n\to(-\infty,+\infty]$ be a proper closed convex function,
	and let $\Omega\subset\mathrm{dom}\,\partial f$ be a compact set on which $f$
	is constant.
	We say that $f$ satisfies the uniform  KL property
	on $\Omega$ if there exist $\eta>0$, $\varepsilon>0$, and a function
	$\varphi:[0,\eta)\to\mathbb{R}_{+}^1$ such that (i) $\varphi(0)=0$; (ii) $\varphi$ is continuous on $[0,\eta)$ and continuously differentiable
		on $(0,\eta)$; (iii) $\varphi$ is concave and $\varphi'(s)>0$ for all $s\in(0,\eta)$;
	and for all $\bar{x}\in\Omega$ and all $x$ satisfying
	$\mathrm{dist}(x,\Omega)<\varepsilon$ and $f(\bar{x})< f(x) < f(\bar{x})+\eta$,
	the following inequality holds:
	$$
	\varphi'\Big(f(x)-f(\bar{x})\Big)\,
	\mathrm{dist}\Big(0,\partial f(x)\Big)\ge 1.
	$$
\end{definition}

Using the definitions on KL properties, we turn to establish our proposed algorithm's convergence. Similar to the work of Liu et al. \cite{liu2019refined}, we define the following function for following theoretical analysis:
\begin{equation}\label{e6}
	\begin{array}{rl}
		 \mathbb{E}_{\tilde\alpha}(z,\tilde z,\tilde \xi)
		 \triangleq& \frac12\|\Phi_{\text{val}} x-b_{\text{val}}\|_2^2+\frac{\tilde\rho}{4}\|z-\tilde z\|_2^{2} +\delta_{\Sigma}(z)\\[2mm]
		& +\tilde\alpha \max \Big\{0, \ \frac12\|\Phi_{\text{tr}} x-b_{\text{tr}}\|_2^2+\langle\tilde\xi, r\rangle+h^*(-\tilde\xi), \ \|\Psi x\|_{1}- r_{1}, \ \sum^{n}\limits_{i=1}\|D_{i}x\|_{1}- r_{2} \Big\},
	\end{array}
\end{equation}
where $h^*(\cdot)$ is a conjugate function of $h(\cdot)$.

Let $\{z^k\}$ be generated by algorithm BF-DCA. We now show the descent property of the sequence $\{\mathbb{E}_{\alpha_k}(z^{k+1},z^k,\xi^k)\}$, which plays a key rule in the subsequent convergence analysis.

\begin{lemma}\label{lem34}
	Let $\{(z^k:=x^k,r^k)\}$ be generated by algorithm BF-DCA. Then, it holds that
	\begin{equation}\label{l23}
		\mathbb{E}_{\alpha_k}(z^k,z^{k-1},\xi^{k-1})\geq \mathbb{E}_{\alpha_k}(z^{k+1},z^{k},\xi^{k})+\frac{\rho_k}{4}\|z^{k+1}-z^k\|_2^2,
	\end{equation}	
and
	\begin{equation}\label{l24}
	\mathrm{dist}\Big(0,\mathbb{E}_{\alpha_k}(z^{k+1},z^{k},\xi^{k})\Big)\leq \frac{\sqrt{2}\rho_k}{2}\|z^k-z^{k-1}\|_2+{(\rho_k+\alpha_k)}\|z^{k+1}-z^k\|_2,
   \end{equation}	
\end{lemma}
\begin{proof}
	(i) Note $z=(x,r)$ and denote
	\begin{equation}\label{gzd}
	g_k(x,r)\triangleq \max \Big\{0, \ \theta_{k}(x,r), \ \|\Psi x\|_{1}- r_{1}, \ \sum^{n}\limits_{i=1}\|D_{i}x\|_{1}- r_{2} \Big\},
	\end{equation}
	where $\theta_{k}(x,r)$ is defined in \eqref{7}. Consequently, the function $\phi_{k}(x,r)$ defined in \eqref{6} can be expressed in the following form:
	\begin{equation}\label{6-1}
		\begin{array}{rl}
			 \phi_{k}(z)= \frac12\|\Phi_{\text{val}} x-b_{\text{val}}\|_2^2+\frac{\rho_k}{2}\|z-z^{k}\|_2^{2} 			 +\alpha_{k} g_k(z),
		\end{array}
	\end{equation}
	It follows from the work of Ye et al. \cite[Lemma 1]{ye2023difference} that
	$$
	\phi_k(z^k)\geq\phi_k(z^{k+1})+\frac{\rho_k}{2}\|z^{k+1}-z^k\|_2^2-\frac{\rho_k}{4}\|z^k-z^{k-1}\|_2^2
	$$
	or, equivalently,
	\begin{equation}\label{l26}
		\phi_k(z^{k+1})\leq\phi_k(z^{k+1})+\frac{\rho_k}{2}\|z^{k+1}-z^k\|_2^2\leq \phi_k(z^{k})+\frac{\rho_k}{4}\|z^k-z^{k-1}\|_2^2.
	\end{equation}
    Then it holds that from \eqref{6-1}
    \begin{equation}\label{l311-1}
    \phi_k(z^{k+1})\leq \frac12\|\Phi_{\text{val}} x^k-b_{\text{val}}\|_2^2 			 +\alpha_{k} g_k(z^k)+\frac{\rho_k}{4}\|z^k-z^{k-1}\|_2^2
    \end{equation}
	Since $h^*(\cdot)$ is a conjugate function of $h(\cdot)$, then we have
	\begin{align*}
		-h(r)=-\sup_{-\xi^{k-1}}\Big\{\langle -\xi^{k-1},r \rangle-h^*(-\xi^{k-1})\Big\}=\inf_{-\xi^{k-1}}\Big\{\langle \xi^{k-1},r \rangle+h^*(-\xi^{k-1})\Big\}
	\end{align*}
	which indicates
	\begin{equation}\label{coneq}
	-h(r^k)\leq \langle \xi^{k-1},r^k \rangle+h^*(-\xi^{k-1}).
	\end{equation}
	Denote
	\begin{equation}\label{tgd}
	\tilde g(z,\tilde\xi)\triangleq\max \Big\{0, \ \frac12\|\Phi_{\text{tr}} x-b_{\text{tr}}\|_2^2+\langle\tilde\xi, r\rangle+h^*(-\tilde\xi), \ \|\Psi x\|_{1}- r_{1}, \ \sum^{n}\limits_{i=1}\|D_{i}x\|_{1}- r_{2} \Big\}.
	\end{equation}
	Then the $\mathbb{E}_{\tilde\alpha}(z,\tilde z,\tilde \xi)$ defined in \eqref{e6} takes the following form
    \begin{equation}\label{e6-2}
		\mathbb{E}_{\tilde\alpha}(z,\tilde z,\tilde \xi)
		= \frac12\|\Phi_{\text{val}} x-b_{\text{val}}\|_2^2+\frac{\tilde\rho}{4}\|z-\tilde z\|_2^{2} +\delta_{\Sigma}(z)+	\tilde\alpha \tilde g(z,\tilde\xi).
    \end{equation}
	
Firstly, based on the inequality \eqref{l311-1}, we aim to establish the following inequality:
\begin{equation}\label{l311-2}
	\frac{1}{2} \|\Phi_{\text{val}} x^k - b_{\text{val}}\|_2^2 + \alpha_k g_k(z^k) + \frac{\rho_k}{4} \|z^k - z^{k-1}\|_2^2 \leq \mathbb{E}_{\alpha_k}(z^k, z^{k-1}, \xi^{k-1}),
\end{equation}
which will demonstrate that $\phi_k(z^{k+1}) \leq \mathbb{E}_{\alpha_k}(z^k, z^{k-1}, \xi^{k-1})$.
Moreover, given that $\delta_{\Sigma}(z^k)=0$,  it is sufficient to prove that
	\begin{align*}
	&\frac12\|\Phi_{\text{val}} x^k-b_{\text{val}}\|_2^2 			 +\alpha_{k} g_k(z^k)+\frac{\rho_k}{4}\|z^k-z^{k-1}\|_2^2 \\[2mm]
	\leq& \frac12\|\Phi_{\text{val}} x^k -b_{\text{val}}\|_2^2+\frac{\rho_{k-1}}{4}\|z^k- z^{k-1}\|_2^{2}+\alpha_{k-1}\tilde g_{k-1}(z^k,\xi^{k-1}).
	\end{align*}
	This can be expressed equivalently as
	$$
	g_k(z^k) \leq \tilde g_{k-1}(z^k,\xi^{k-1}).
	$$
	This holds because the sequences $\{\alpha_k\}$ and $\{\rho_k\}$ are non-decreasing, while both $g_k(\cdot)$ and $\tilde g_{k-1}(\cdot, \cdot)$ are non-negative functions.
    From the definition of $g_k(z)$ in \eqref{gzd} and the definition of $\tilde g(z,\xi)$ in \eqref{tgd}, we get that it is equivalent to proving that
    \begin{equation}\label{7-11}
	\theta_{k}{(x^k,r^k) = \frac{1}{2}\|\Phi_{\text{tr}} x^k-b_{\text{tr}}\|^{2}_{2} -h(r^k)+\langle\xi^{k},r^k-r^{k}\rangle}\leq \frac12\|\Phi_{\text{tr}} x^k-b_{\text{tr}}\|_2^2+\langle\xi^{k-1}, r^k\rangle+h^*(-\xi^{k-1}),
    \end{equation}
   or
   $$
   	-h(r^k)\leq \langle\xi^{k-1}, r^k\rangle+h^*(-\xi^{k-1}).
   $$
   This is exactly the inequality \eqref{coneq}, which subsequently proves the desired result
   \begin{equation}\label{phie}
  \phi_k(z^{k+1})\leq \mathbb{E}_{\alpha_k}(z^k,z^{k-1},\xi^{k-1}).
  \end{equation}

   Secondly, from the fact of $(\bar x^k,\xi^k)\in\mathcal{M}(r^k)$ and the definition of $\mathcal{M}(r)$ in \eqref{6}, we know that $-\xi^k\in\partial h(r^k)$. Note that $h(r)$ is a convex function, we get $h(r)\geq h(r^k)+\langle -\xi^k,r-r^k \rangle$, i.e., $\langle -\xi^k,r \rangle-h(r)\leq \langle -\xi^k,r^k \rangle-h(r^k)$. Furthermore, based on the definition of the conjugate function, we have $  h^*(-\xi^k) = \sup_r \left\{ \langle -\xi^k, r \rangle - h(r) \right\}$. Thus, it follows that
   \begin{equation}\label{l311}
   h^*(-\xi^k) = \langle -\xi^k, r^k \rangle - h(r^k)
   \end{equation}
   This means that
   $$
   \frac{1}{2}\|\Phi_{\text{tr}} x^{k+1}-b_{\text{tr}}\|^{2}_{2}+\langle\xi^k,r^{k+1} \rangle+h^*(-\xi^k) = \frac{1}{2}\|\Phi_{\text{tr}} x^{k+1}-b_{\text{tr}}\|^{2}_{2}+\langle \xi^k, r^{k+1}-r^k \rangle - h(r^k).
   $$
   i.e., $\tilde g_k(z^{k+1},\xi^k)=g_k(z^{k+1})$. From the definition of $\phi_{k}(z^{k+1})$ in \eqref{6-1} and $\mathbb{E}_{\alpha_k}(z^k, z^{k-1}, \xi^{k-1})$ in \eqref{e6-2}, and also note $z^{k+1}\in\Sigma$, we get
   $$
   \phi_{k}(z^{k+1})=\mathbb{E}_{\alpha_k}(z^{k+1}, z^{k}, \xi^{k})+\frac{\rho_k}{4}\|z^{k+1}-z^k\|_2^2.
   $$
   It follow from \ref{phie} that
   $$
   \mathbb{E}_{\alpha_k}(z^{k+1}, z^{k}, \xi^{k})+\frac{\rho_k}{4}\|z^{k+1}-z^k\|_2^2\leq \mathbb{E}_{\alpha_k}(z^k,z^{k-1},\xi^{k-1}),
   $$
   which yields the desired result \eqref{l23} of this lemma.

   (ii) We now turn to prove the inequality given in \eqref{l24}. It is easy to deduce that the Clarke sub-differentiable \cite{clarke1990optimization} of $g_k(x,r)$ in \eqref{gzd}
    \begin{equation}
   	\partial g_k(x, r) = \left\{
   	\begin{array}{l}
   		\eta_0 \partial \theta_k(x, r) + \eta_1 \partial \|\Psi x\|_1 + \eta_2 \partial  \sum_{i=1}^n \|D_i x\|_1  \\[2mm]
   		\text{subject \ to:} \\[2mm]
   		\eta_0, \eta_1, \eta_2 \in [0, 1], \ \eta_0 + \eta_1 + \eta_2 \leq 1, \\[2mm]
   				(1-\eta_0-\eta_1-\eta_2)  g_k(x, r)  = 0, \\[2mm]
   		\eta_0 ( \theta_k(x, r) - g_k(x, r) ) = 0, \\[2mm]
   		\eta_1 ( \|\Psi x\|_1 - r_1 - g_k(x, r) ) = 0, \\[2mm]
   		\eta_2 ( \sum_{i=1}^n \|D_i x\|_1 - r_2 - g_k(x, r) ) = 0,
   	\end{array}
   	\right.
   \end{equation}
   where $\partial\theta_k(x,r)=(\Phi_{\text{tr}}^\top(\Phi_{\text{tr}} x-b_{\text{tr}}),\xi^k)\in\mathbb{R}^{n+2}$.
   Moreover, the Clarke sub-differentiable \cite{clarke1990optimization} of $\tilde g_k(z,\xi)$ in \eqref{tgd} can also be expressed as
   \begin{equation}
   	\partial_z \tilde{g}(z,\xi^k) = \left\{
   	\begin{array}{l}
   		\eta_0 \partial \big( \frac{1}{2}\|\Phi_{\text{tr}} x - b_{\text{tr}}\|_2^2 + \langle \xi^k, r \rangle + h^*(-\xi^k) \big) + \eta_1 \partial \|\Psi x\|_1 + \eta_2 \partial ( \sum_{i=1}^n \|D_i x\|_1 ) \\[2mm]
   			\text{subject  to:} \\[2mm]
   			\eta_0, \eta_1, \eta_2 \in [0, 1], \eta_0 + \eta_1 + \eta_2 \leq 1,\\[2mm]
   			(1-\eta_0-\eta_1-\eta_2) \tilde{g}(z,\xi^k) = 0, \\[2mm]
   			\eta_0 \left( \frac{1}{2}\|\Phi_{\text{tr}} x - b_{\text{val}}\|_2^2 + \langle \xi^k, r \rangle + h^*(-\xi^k) - \tilde{g}(x, r) \right) = 0, \\[2mm]
   			\eta_1 ( \|\Psi x\|_1 - r_1 - \tilde{g}(z,\xi^k) ) = 0, \\[2mm]
   			\eta_2 ( \sum_{i=1}^n \|D_i x\|_1 - r_2 - \tilde{g}(z,\xi^k) ) = 0,
   	\end{array}
   	\right.
   \end{equation}
   and that
   \begin{equation}
   	\partial_\xi \tilde{g}(z^{k+1}, \xi) = \left\{
   	\begin{array}{l}
   		\eta ( r^{k+1} - \partial h^*(-\xi) ) \\ [2mm]
   		\text{subject to:}\\[2mm]
   			 \eta \in [0, 1], \
   			\eta \big( \frac{1}{2}\|\Phi_{\text{tr}} x^{k+1} - b_{\text{tr}}\|_2^2 + \langle \xi, r^{k+1} \rangle + h^*(-\xi) - \tilde{g}(x^{k+1}, r^{k+1}, \xi) \big) = 0.
   	\end{array}
   	\right.
   \end{equation}
   From the definition of $\theta_k(x,r)$ in \eqref{7} and the relation \eqref {l311}  we get that
   \begin{align*}
   	&\theta_{k}{(x^{k+1},r^{k+1}) =\dfrac{1}{2}\|\Phi_{\text{tr}} x^{k+1}-b_{\text{tr}}\|^{2}_{2} -h(r^k)+\langle\xi^{k},r^{k+1}-r^{k}\rangle}\\[2mm]
   	=&\dfrac{1}{2}\|\Phi_{\text{tr}} x^{k+1}-b_{\text{tr}}\|^{2}_{2}+\langle \xi^k, r^{k+1} \rangle + h^*(-\xi^k),
   \end{align*}
   which means $g_k(z^{k+1})=\tilde g_k(z^{k+1},\xi^k)$ and that
   \begin{equation}\label{pareq}
        \partial_z g_k(z^{k+1})=\partial_z\tilde g_k(z^{k+1},\xi^k).
   \end{equation}
   On the other hand, from \eqref{6-1} we get that
   $$
   \partial\phi_k(z)=\Phi_{\text{val}}^\top(\Phi_{\text{val}} x-b_{\text{val}})\times\{0\}_2+\rho_k(z-z^k)+\alpha_k\partial_z g_k(z)
   $$
   and from \eqref{e6-2} that
   $$
   \partial_z\mathbb{E}_{\alpha_k}(z,z^k,\xi^k)=\Phi_{\text{val}}^\top(\Phi_{\text{val}} x-b_{\text{val}})\times\{0\}_2+\frac{\rho_k}{2}(z-z^k)+\mathbb{N}_{\Sigma}(z)+\alpha_k\partial_z\tilde g_k(z,\xi^k).
   $$
   Hence
   \begin{equation}\label{e37}
   	   \partial_z\mathbb{E}_{\alpha_k}(z,z^k,\xi^k)=\partial\phi_k(z)-\frac{\rho_k}{2}(z-z^k)+\mathbb{N}_{\Sigma}(z).
   \end{equation}
   Since $\mathbb{E}_{\alpha_k}(z,\tilde z,\tilde\xi)$ is a function with respect to $z$, $\tilde z$, and $\tilde \xi$, hence its subdifferentiable can be expressed in the following manner for each variable:
   \begin{align*}
   \partial \mathbb{E}_{\alpha_k}(z^{k+1},z^k,\xi^k)=& \partial_z \mathbb{E}_{\alpha_k}(z^{k+1},z^k,\xi^k) \times
   \partial_{\tilde z} \mathbb{E}_{\alpha_k}(z^{k+1},z^k,\xi^k)\times  \partial_{\tilde\xi} \mathbb{E}_{\alpha_k}(z^{k+1},z^k,\xi^k)\\[2mm]
   =& \partial_z \mathbb{E}_{\alpha_k}(z^{k+1},z^k,\xi^k) \times
   \Big\{\frac{\rho_k}{2}(z^k-z^{k+1})\Big\}\times  \partial_{\tilde\xi} \mathbb{E}_{\alpha_k}(z^{k+1},z^k,\xi^k)\\[2mm]
   \supseteq & \partial_z \mathbb{E}_{\alpha_k}(z^{k+1},z^k,\xi^k) \times
   \Big\{\frac{\rho_k}{2}(z^k-z^{k+1})\Big\}\times \alpha_k \eta_{k}( r^{k+1} - \partial h^*(-\xi^k) ).
   \end{align*}
   Since $z^{k+1}$ is an approximated solution of \eqref{6} such that the specified tolerance \eqref{9} is satisfied, then there exists a error vector $e_k$ such that
   $$
   e_k\in\partial\phi_k(z^{k+1})+\mathbb{N}_{\Sigma}(z^{k+1}),
   $$
   that is to say
   \begin{equation}\label{sqrt2}
     \|e_k\|_2\leq \frac{\sqrt{2}}{2}\rho_k \|z^{k}-z^{k-1}\|.
   \end{equation}
   Then, from \eqref{e37}, it is yields that
   $$
   e_k-\frac{\rho_k}{2}(z^{k+1}-z^k)\in \partial_z\mathbb{E}_{\alpha_k}(z^{k+1},z^k,\xi^k).
   $$
   It is known from \cite[Theorem 23.5]{rockafellar2015convex} that $-\xi^k\in\partial h(r^k)$ is equivalent to $r^k\in\partial h^*(-\xi^k)$, we get
   \begin{equation}
   	\mathcal{X}\triangleq
   	\left(
   	   \begin{array}{c}
   		e_k-\frac{\rho_k}{2}(z^{k+1}-z^k)\\[2mm]
   		\frac{\rho_k}{2}(z^k-z^{k+1})\\[2mm]
   		\alpha_k\eta_{k}(r^{k+1}-r^k)
   	   \end{array}
      \right)	\in \partial\mathbb{E}_{\alpha_k}(z^{k+1},z^k,\xi^k).
   \end{equation}
   Therefore, it is adequate to demonstrate the subsequent inequality in order to validate the assertion  \eqref{l24}:
   	\begin{equation}\label{l24-2}
   	\mathrm{dist}\Big(0,\mathcal{X}\Big)\leq \frac{\sqrt{2}\rho_k}{2}\|z^k-z^{k-1}\|_2+{(\rho_k+\alpha_k)}\|z^{k+1}-z^k\|_2.
   \end{equation}	
  It is easy to show that
  \begin{align*}
  	\mathrm{dist}\Big(0,\mathcal{X}\Big)&\leq\| e_k-\frac{\rho_k}{2}(z^{k+1}-z^k)\|_2+\frac{\rho_k}{2}\|(z^k-z^{k+1})\|_2+\alpha_k\eta_{k}\|(r^{k+1}-r^k)\|_2\\[2mm]
  	&\leq\| e_k\|_2+{\rho_k}\|(z^{k+1}-z^{k})\|_2+\alpha_k\eta_{k}\|(r^{k+1}-r^k)\|_2.
  \end{align*}
  Considering \eqref{sqrt2} and the fact that $\eta_k \in [0, 1]$, along with the inequality $\|r^{k+1} - r^k\|_2 \leq \|z^{k+1} - z^k\|_2$, we can derive \eqref{l24-2}. This demonstrates that the assertion in \eqref{l24} holds true, thereby completing the proof.

\end{proof}

Finally, we are ready to establish the convergence of the BF-DCA algorithm. We first observe that all the terms involved in
$\mathbb{E}_{\tilde\alpha}(z,\tilde z,\tilde \xi)$ are semi-algebraic. Since the class of semi-algebraic functions is closed under finite sums and
finite maximum operations, it follows that
$\mathbb{E}_{\tilde\alpha}(\cdot,\cdot,\cdot)$ is itself a semi-algebraic function, consequently, it is satisfies the
KL property.

\begin{theorem}
Let the sequence $\{(z^k := x^k, r^k)\}$ be generated by the BF-DCA algorithm. Then, the sequence $\{z^k\}$ is bounded.
 Suppose that  there exists a $\delta>0$ such that $r_1^k\geq \delta$ and $r_2^k\geq\delta$, then the sequence $\{z^k\}$ converges to a KKT point of problem \eqref{blp2}.
\end{theorem}
\begin{proof}
	To establish the convergence of the sequence $\{z^k\}$, it is sufficient to show that $\sum_{k=1}^\infty \|z^{k+1} - z^k\|_2<\infty$.
	
	From the steps of algorithm BF-DCA, we see that $\alpha_k\to\alpha_{\max}$ as $k\to\infty$ and $\rho_k\to\rho_{\max}$ as $k\to\infty$. From the definition of $\mathbb{E}_{\tilde\alpha}(z,\tilde z,\tilde \xi)$ in \eqref{e6}, we see that the sequence $\{\mathbb{E}_{\tilde\alpha}(z^{k+1},z^k,\xi^k)\}$ is bounded below.  Then from \eqref{l23} in Lemma \ref{lem34} we get  i.e., $\|z^{k+1}-z^k\|_2^2\to 0$ and $k\to\infty$.
	
	Let $\mathbb{C}$ be a set that contains all the cluster points of the sequence $\{(z^k,z^{k-1},\xi^k)\}$. Clearly, $\mathbb{C}$ is a closed set and that $\mathrm{dist((z^k,z^{k-1},\xi^{k-1}),0)}\to 0$ as $k\to\infty$. It follows from Theorem \ref{them32} that any cluster point of the sequence $\{z^k\}$ is a KKT point of the problem described in \eqref{blp2}, and its corresponding point belongs to the set $\mathbb{C}$.
	Because it has been known that $h(r)$ given in \eqref{3} is Lipschitz continuous on $\mathbb{R}^2_+$, then $\partial h(r)$ is bounded on the interior of $\mathbb{R}^2_+$. On the other hand, since it is assumed that there exists a constant $\delta > 0$ such that $r_1^k \geq \delta$ and $r_2^k \geq \delta$, and that $-\xi^k \in \partial h(r^k)$, it follows that the sequence $\{\xi^k\}$ is bounded. Therefore, the sequence $\{(z^k, z^{k-1}, \xi^k)\}$ is also bounded, meaning that the set $\mathbb{C}$ is bounded. Furthermore, since $\mathbb{C}$ is a closed set, it is must be a compact set.
	
	From \eqref{l23} in Lemma \ref{lem34} and $\|z^{k+1}-z^k\|_2\to 0$, we know that
	$\{\mathbb{E}_{\alpha_{k-1}}(z^k,z^{k-1},\xi^{k-1})\}$ is a bounded below and nonincreasing sequence. Hence, for any subsequence $\{\mathbb{E}_{\alpha_{k_j-1}}(z^{k_j},z^{k_j-1},\xi^{k_j-1})\}$, it holds that
	$$
	\bar{\mathbb{E}}\triangleq \lim_{j\to\infty} \mathbb{E}_{\alpha_{k_j-1}}(z^{k_j},z^{k_j-1},\xi^{k_j-1})=\lim_{k\to\infty}\mathbb{E}_{\alpha_{k-1}}(z^k,z^{k-1},\xi^{k-1}),
	$$
	which shows that $\bar{\mathbb{E}}$ is the function value of $\mathbb{E}(\cdot,\cdot,\cdot)$ on the compact set $\mathbb{C}$.
	
	Because $\lim_{k\to\infty}\mathrm{dist}((z^k,z^{k-1},\xi^{k-1}),\mathbb{C})=0$ and  $\lim_{k\to\infty}\mathbb{E}_{\alpha_{k-1}}(z^k,z^{k-1},\xi^{k-1})=\bar{\mathbb{E}}$, then for any $\varepsilon_1>0$ and $\varepsilon_2>0$, there exist a $k_0$ such that for any $k\geq k_0$, it holds that $\mathrm{dist}((z^k,z^{k-1},\xi^{k-1}),\mathbb{C})<\varepsilon_1$ and $\bar{\mathbb{E}}<\mathbb{E}_{\alpha_{k-1}}(z^k,z^{k-1},\xi^{k-1})<\bar{\mathbb{E}}+\varepsilon_2$.
	Because $\mathbb{E}_{\tilde{\alpha}}(z,\tilde z,\tilde\xi)$ is a KL function also a constant on $\mathbb{C}$, from Definition \ref{dfkl}, we know that there is continuous concave function $\varphi(\cdot)$ such that for any $k\geq k_0$, it holds that
    $$
    \varphi'\Big(\mathbb{E}_{\alpha_{k-1}}(z^k,z^{k-1},\xi^{k-1})-\bar{\mathbb{E}}\Big)\mathrm{dist}\Big(0,\partial \mathbb{E}_{\alpha_{k-1}}(z^k,z^{k-1},\xi^{k-1})\Big)\geq 1.
    $$
	or, from \eqref{l24} that
	\begin{equation}\label{e39}
	\varphi'\Big(\mathbb{E}_{\alpha_{k-1}}(z^k,z^{k-1},\xi^{k-1})-\bar{\mathbb{E}}\Big)\bigg(\frac{\sqrt{2}\rho_{k-1}}{2}\|z^{k-1}-z^{k-2}\|_2^2+{(\rho_{k-1}+\alpha_{k-1})}\|z^{k}-z^{k-1}\|_2^2\bigg)\geq 1.
	\end{equation}
	Since $\varphi'(\cdot)$ is non-increasing from the fact that $\varphi(\cdot)$ is concave, then it holds that
	\begin{align*}
		&\varphi\Big(\mathbb{E}_{\alpha_{k-1}}(z^k,z^{k-1},\xi^{k-1})-\bar{\mathbb{E}}\Big)-\varphi\Big(\mathbb{E}_{\alpha_{k}}(z^{k+1},z^{k},\xi^{k})-\bar{\mathbb{E}}\Big)\\[2mm]
		\geq &\varphi'\Big(\mathbb{E}_{\alpha_{k-1}}(z^k,z^{k-1},\xi^{k-1})-\bar{\mathbb{E}}\Big)\bigg(\mathbb{E}_{\alpha_{k-1}}(z^k,z^{k-1},\xi^{k-1})- \mathbb{E}_{\alpha_{k}}(z^{k+1},z^{k},\xi^{k})\bigg).
	\end{align*}
	Combining it with \eqref{l23} in Lemma \ref{lem34}, we get
	\begin{align}
		&\varphi'\Big(\mathbb{E}_{\alpha_{k-1}}(z^k,z^{k-1},\xi^{k-1})-\bar{\mathbb{E}}\Big)\cdot\frac{\rho_k}{4}\|z^{k+1}-z^k\|_2^2\nonumber\\[2mm]
		\leq& \varphi'\Big(\mathbb{E}_{\alpha_{k-1}}(z^k,z^{k-1},\xi^{k-1})-\bar{\mathbb{E}}\Big)\cdot \bigg(\mathbb{E}_{\alpha_k}(z^k,z^{k-1},\xi^{k-1})- \mathbb{E}_{\alpha_k}(z^{k+1},z^{k},\xi^{k}) \bigg)\nonumber\\[2mm]
		\leq & \varphi\Big(\mathbb{E}_{\alpha_{k-1}}(z^k,z^{k-1},\xi^{k-1})-\bar{\mathbb{E}}\Big)-\varphi\Big(\mathbb{E}_{\alpha_{k}}(z^{k+1},z^{k},\xi^{k})-\bar{\mathbb{E}}\Big).\label{e40}
	\end{align}
	Substituting \eqref{e39} into \eqref{e40}, we get
	\begin{align}
	&\bigg(\varphi\Big(\mathbb{E}_{\alpha_{k-1}}(z^k,z^{k-1},\xi^{k-1})-\bar{\mathbb{E}}\Big)-\varphi\Big(\mathbb{E}_{\alpha_{k}}(z^{k+1},z^{k},\xi^{k})-\bar{\mathbb{E}}\Big) \bigg)\nonumber\\[2mm]
	\cdot&
	\bigg(\frac{\sqrt{2}\rho_{k-1}}{2}\|z^{k-1}-z^{k-2}\|_2+{(\rho_{k-1}+\alpha_{k-1})}\|z^{k}-z^{k-1}\|_2\bigg)\geq \frac{\rho_k}{4}\|z^{k+1}-z^k\|_2^2.\label{e41}
	\end{align}
	On the other hand, because the second term in the left-hand side satisfies
	\begin{align*}
		&\frac{\sqrt{2}\rho_{k-1}}{2}\|z^{k-1}-z^{k-2}\|_2+{(\rho_{k-1}+\alpha_{k-1})}\|z^{k}-z^{k-1}\|_2\\[2mm]
	   =&	(\rho_{k-1}+\alpha_{k-1})\bigg(\frac{\sqrt{2}\rho_{k-1}}{2(\rho_{k-1}+\alpha_{k-1})}\|z^{k-1}-z^{k-2}\|_2+\|z^{k}-z^{k-1}\|_2 \bigg)\\[2mm]
	   \leq & (\rho_{k-1}+\alpha_{k-1})\bigg( \|z^{k-1}-z^{k-2}\|_2+\|z^{k}-z^{k-1}\|_2 \bigg),
	\end{align*}
	then it follows that
		\begin{align*}
		&\frac{4(\rho_{k-1}+\alpha_{k-1})}{\rho_k}\bigg(\varphi\Big(\mathbb{E}_{\alpha_{k-1}}(z^k,z^{k-1},\xi^{k-1})-\bar{\mathbb{E}}\Big)-\varphi\Big(\mathbb{E}_{\alpha_{k}}(z^{k+1},z^{k},\xi^{k})-\bar{\mathbb{E}}\Big) \bigg)\nonumber\\[2mm]
		\cdot&
		\bigg( \|z^{k-1}-z^{k-2}\|_2+\|z^{k}-z^{k-1}\|_2 \bigg)\geq \|z^{k+1}-z^k\|_2.
	\end{align*}
	Taking square-root on the both sides of the above inequality and using the relation $2\sqrt{ab}\leq a+b$, we gets that
	\begin{align*}
		4\|z^{k+1}-z^k\|_2\leq& \Big( \|z^{k-1}-z^{k-2}\|_2+\|z^{k}-z^{k-1}\|_2 \Big)\\[2mm]
		&+\frac{16(\rho_{k-1}+\alpha_{k-1})}{\rho_k}\bigg(\varphi\Big(\mathbb{E}_{\alpha_{k-1}}(z^k,z^{k-1},\xi^{k-1})-\bar{\mathbb{E}}\Big)-\varphi\Big(\mathbb{E}_{\alpha_{k}}(z^{k+1},z^{k},\xi^{k})-\bar{\mathbb{E}}\Big) \bigg).
	\end{align*}
	For convenience of the following analysis, we set $\alpha_k\equiv\alpha_{\max}$ and $\rho_k\equiv\rho_{\max}$ for each $k\geq k_0$ with a given $k_0$. Adding both sides of the above inequality from $i=k_0$ to $k$, we get
	\begin{align}
		\sum_{i=k_0}^k 4\|z^{i+1}-z^i\|_2\leq& \sum_{i=k_0}^k\Big( \|z^{i-1}-z^{i-2}\|_2+\|z^{i}-z^{i-1}\|_2 \Big)\nonumber\\
		&+\frac{16(\rho_{\max}+\alpha_{\max})}{\rho_{\max}}\bigg(\varphi\Big(\mathbb{E}_{\alpha_{\max}}(z^{k_0},z^{k_0-1},\xi^{k_0-1})-\bar{\mathbb{E}}\Big)-\varphi\Big(\mathbb{E}_{\alpha_{\max}}(z^{k+1},z^{k},\xi^{k})-\bar{\mathbb{E}}\Big) \bigg).\label{e43}
	\end{align}
	On the other hand, we have
		\begin{align*}
		&\sum_{i=k_0}^k 4\|z^{i+1}-z^i\|_2- \sum_{i=k_0}^k\Big( \|z^{i-1}-z^{i-2}\|_2+\|z^{i}-z^{i-1}\|_2 \Big)\\
		=&\sum_{i=k_0}^k 2\|z^{i+1}-z^i\|_2- \sum_{i=k_0}^k\Big( \|z^{i-1}-z^{i-2}\|_2+\|z^{i}-z^{i-1}\|_2 -2\|z^{i+1}-z^i\|_2\Big)\\[2mm]
		=&\sum_{i=k_0}^k 2\|z^{i+1}-z^i\|_2-\Big(2\|z^{k_0}-z^{k_0-1}\|_2+\|z^{k_0-1}-z^{k_0-2}\|_2 \Big)+\Big(2\|z^{k+1}-z^{k}\|_2+3\|z^{k}-z^{k-1}\|_2 \Big).
	\end{align*}
	From the fact that $\varphi(\cdot)>0$, we get that the inequality \eqref{e43} can further be reformulated as
	\begin{align}
	    \sum_{i=k_0}^k 2\|z^{i+1}-z^i\|_2\leq& \Big(2\|z^{k_0}-z^{k_0-1}\|_2+\|z^{k_0-1}-z^{k_0-2}\|_2 \Big)\nonumber\\[2mm]
	    &+\frac{16(\rho_{\max}+\alpha_{\max})}{\rho_{\max}}\varphi\Big(\mathbb{E}_{\alpha_{\max}}(z^{k_0},z^{k_0-1},\xi^{k_0-1})-\bar{\mathbb{E}}\Big).\label{e44}
	\end{align}
Since it is assumed that
$\mathbb{E}_{\alpha_{k_0}}(z^{k_0}, z^{k_0-1}, \xi^{k_0-1}) \to \bar{\mathbb{E}}$
as $k_0 \to \infty$, and it has been shown that
$\|z^{k_0+1} - z^{k_0}\|_2 \to 0$ and
$\|z^{k_0-1} - z^{k_0-2}\|_2 \to 0$ as $k_0 \to \infty$,
it follows from \eqref{e44} that
$$
\sum_{i=k_0}^{k} \|z^{i+1} - z^{i}\|_2 < \infty.
$$
Consequently,
$$
\sum_{k=1}^{\infty} \|z^{k+1} - z^{k}\|_2 < \infty.
$$
Hence, $\{z^k\}$ is a Cauchy sequence and therefore converges.
\end{proof}

%%%%%%%%%%%%%%%%%%%%%%%%%%%%%%%%%%%%%%%%%%%%%%%%%%%%%%%%%%%%%%%%%
\section{Numerical experiments on some  images}\label{sec4}
\setcounter{equation}{0}
%%%%%%%%%%%%%%%%%%%%%%%%%%%%%%%%%%%%%%%%%%%%%%%%%%%%%%%%%%%%%%%%
In this section, we evaluate the practical performance of the BF-DCA algorithm for image restoration tasks and compare it with other parameterized approaches. All the tested algorithms are implemented in Python and executed on a personal computer equipped with an Intel Core i7 processor at 1.80 GHz and 8 GB of RAM.
The benchmark methods used for performance comparison are summarized as follows:
\begin{itemize}
	\item[-] \textbf{Grid Search (GS)}: A brute-force strategy that exhaustively explores a predefined uniform hyperparameter grid by evaluating all possible hyperparameter combinations.
	\item[-] \textbf{Random Search (RS)}: A stochastic approach that uniformly samples $30$ hyperparameter configurations for the elastic-net model and $50$ configurations for the generalized elastic-net model, followed by performance evaluation.
	\item[-] \textbf{TPE}: A Bayesian optimization method based on the Tree-structured Parzen Estimator, which models the objective function using Gaussian mixture models; see Bergstra et al.~\cite[pp.~115--123]{B2013}.
\end{itemize}

Although the lower-level optimization problem \eqref{3} can be efficiently solved by ADMM, we adopt a PyTorch-based implementation to solve it in our experiments.
Additionally, to ensure a fair comparison, all subproblems arising from the GS and RS frameworks are also solved using the same PyTorch-based implementation.
Finally, the implementation of the TPE method used in this study is publicly available at
\url{https://github.com/hyperopt/hyperopt}.

\subsection{General descriptions}

In this evaluation, we assess the practical performance of all algorithms using actual brain tumor images sized at \(256 \times 256\) sourced from the BSDS300 library, which is publicly accessible at \url{https://www.kaggle.com/datasets/thomasdubail/brain-tumors-256x256}. For each real image, we rescale pixel values from integers in $[0,255]$ to real numbers in $[0,1]$.
In \eqref{orig}, the linear operator $\Phi$ is chosen to be a partial Fourier operator, and the noise is assumed to be random noise and impulse noise.
Here, the $k$-space undersampling is simulated using a pseudo-random sampling mask, which samples the Fourier domain along multiple radial lines emanating from the center.
The Haar wavelet transform matrix $\Psi$ is constructed using the open-source PyWavelets library in Python, which provides functions for one-dimensional, two-dimensional, and three-dimensional wavelet transforms, as well as filter bank design, analysis, and processing.
For the considered noise $\epsilon\in\mathbb{R}^m$, we first define "$\text{min}$" and "$\text{max}$" as the minimum and maximum values of $\Phi {x}^\star$, respectively, where $x^\star$ is the ground-truth.
In the salt-and-pepper noise case, each $\epsilon_i$ is randomly set to either "$\text{min}$" or "$\text{max}$",
whereas in the random noise case, $\epsilon_i$ is randomly drawn from the interval "$[\text{min}, \text{max}]$". In addition, we set the noise level as $0.01$ in the salt-and-pepper noise case.

In this test, we evaluate the performance of each algorithm using three quantitative metrics, namely:
\begin{itemize}
	\item[(i)] Relative $\ell_2$-norm error (RLNE) to measure the overall relative error between the reconstructed solution and the ground truth:
	$$
	\mathrm{RLNE} \triangleq \frac{\|\bar x - x^\star\|_2}{\|\bar x\|_2};
	$$	
	\item[(ii)] Peak signal-to-noise ratio (PSNR) to measure the discrepancy between the restored image and the ground truth:
	$$
	\mathrm{PSNR} \triangleq 20 \log_{10}\!\left( \frac{\sqrt{n}}{\|\bar x - x^\star\|_2} \right)\ \text{(dB)};
	$$	
	\item[(iii)] Normalized relative error (NRE) to quantify the relative residual associated with the model constraint or data consistency term:
	$$
	\mathrm{NRE} \triangleq \frac{\|\Psi x^\star - b\|_2}{1 + \|b\|_2}.
	$$
\end{itemize}

In this test, we set $c_{\alpha} = 1$, and initialize $\alpha_0 = 0$, with an increment of $\delta_{\alpha} = 0.005$ until it reaches its maximum value $\alpha_{\max}=10$. We set $ \rho_k $ to be a constant, $ \rho_k = 0.001 $, i.e., $ \rho_0 = \rho_{\max} = 0.001 $, with an increment of $ \delta_{\rho} = 0 $.
For the algorithm's initialization, we set the starting point to $ x^0 = 0_n $ and $ r^0 = (0.1, 0.5)^\top $, where $ 0_n $ denotes a zero vector of dimension $ n $.

\subsection{Test on two single images}

In this section, we assess the performance of the BF-DCA algorithm using a phantom   Shepp-Logan image of size $64 \times 64$ and a brain MRI image of size $128 \times 128$. Additionally, we compare the performance of BF-DCA with the GS, GS, and TPE algorithms based on the aforementioned metrics.
Since only a single image is used in this test, we did not distinguish between the training set and the validation set, i.e., $ \Phi_{\text{val}} = \Phi_{\text{tr}} $ and $ b_{\text{val}} = b_{\text{tr}} $.

In first test on phantom image, we use a $57\%$ sampling rate which means that the  the observed data $b\in\mathbb{R}^m$ is derived using the generating process \eqref{orig}  with $m=2335\gets 64\times 64\times 57\%$ from the ground-truth $x^\star\in\mathbb{R}^{n}$ with $n=4096\gets 64\times 64$.
The ground truth and its undersampled image with added salt-and-pepper noise are shown in (a) and (b) of Figure \ref{fig:graph1}, respectively.
We apply the proposed BF-DCA algorithm, along with other comparison algorithms, to reconstruct the original image from the observed data $ b$, and present the restored results in (c) to (f) of Figure \ref{fig:graph1}.
By comparing the images in (c), (d), (e), and (f), we can clearly observe that all the algorithms produce acceptable restored images. However, the image in (f) appears to be slightly better, showing a more accurate restoration.
Additionally, the RLNE value for the image restored by BF-DCA is $5.40\%$, which is slightly smaller than the $6.10\%$ achieved by GS, the $6.13\%$ by RS, and the $7.04\%$ by TPE.
In summary, this toy test demonstrates that the BF-DCA algorithm, based on the bilevel programming problem \eqref{blp2}, outperforms other approaches in terms of restoring a phantom image.

\begin{figure}[h]
	\centering
	\includegraphics[width=0.18\textwidth]{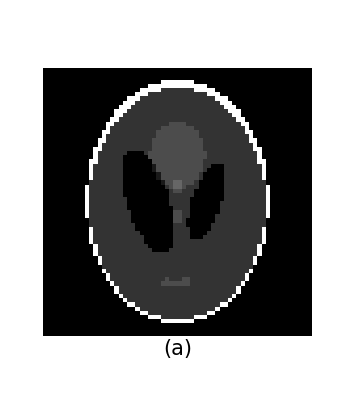}\hspace{-.5cm}
	\includegraphics[width=0.18\textwidth]{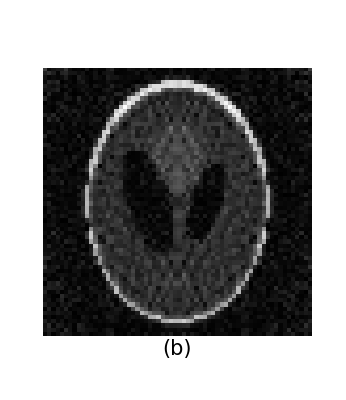}\hspace{-.5cm}
	\includegraphics[width=0.18\textwidth]{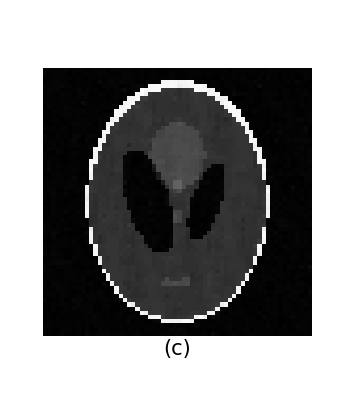}\hspace{-.5cm}
	\includegraphics[width=0.18\textwidth]{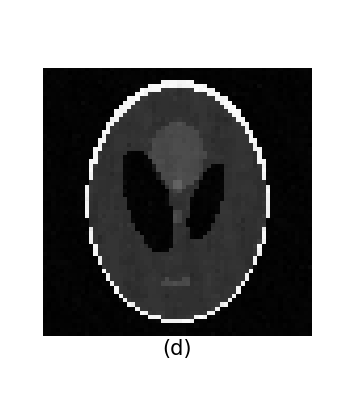}\hspace{-.5cm}
	\includegraphics[width=0.18\textwidth]{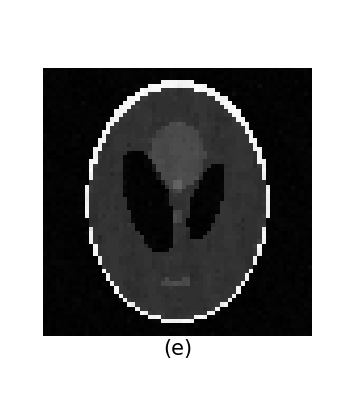}\hspace{-.5cm}
	\includegraphics[width=0.18\textwidth]{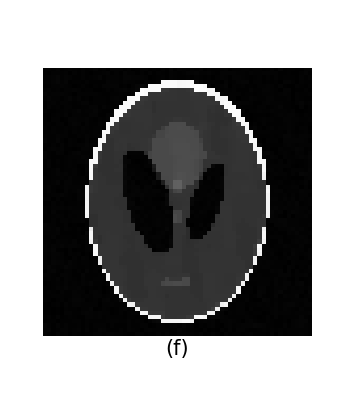}
	\vspace{-0.5cm}
	\caption{The ground-truth phantom image (a),  the undersampled image with added noise (b), and the restored images by algorithms GS (c), RS (d), TPE (e), and IR-DCA (f).}
	\label{fig:graph1}
\end{figure}

In the second test, we evaluate the convergence behavior of each algorithm by analyzing how the PSNR and RLNE values change as the computing time and the number of iterations increase.
In algorithms of GS, RS and TPE, we  select values for $ u_1 $ and $ u_2 $ within the interval $ [-9, -3] $, and use the relations $ u_1 = \log_{10}(\lambda_1) $ and $ u_2 = \log_{10}(\lambda_2) $ to determine the values of $ \lambda_1 $ and $ \lambda_2 $.
In the case of the GS algorithm, a hyperparameter search is performed over a $14 \times 14$ uniform grid within the interval $[-9, -3]$, yielding a total of $196$ combinations of $(u_1, u_2)$. Conversely, for the RS algorithm, $200$ pairs of values for $u_1$ and  $u_2)$ are selected randomly and uniformly from the interval $[-9, -3]$. In the case of TPE, $u_1$ and $u_2$ are randomly chosen from $200$ points within the interval $[-9, -3]$, using a probabilistic model to estimate their distribution.
To clearly illustrate the convergence behavior of each algorithm, we plot the decreasing curves of the PSNR and RLNE values as a function of iterations and computation time in Figure \ref{fig:graph2}.
From plots (b) and (d) in Figure \ref{fig:graph2}, we observe that the RLNE values corresponding to all algorithms decrease monotonically, indicating that all the tested algorithms are capable of restoring high-quality images from noisy observations.
We also observe that the red curve corresponding to the BF-DCA algorithm lies at the bottom, which indicates that BF-DCA outperforms the other algorithms by requiring less computation time and fewer iterations.
From plots (a) and (c), we observe that the red line is positioned at the top, indicating that BF-DCA is the most effective method for enhancing the signal-to-noise ratio.

\begin{figure}[h]
	\centering
	\includegraphics[height=8cm]{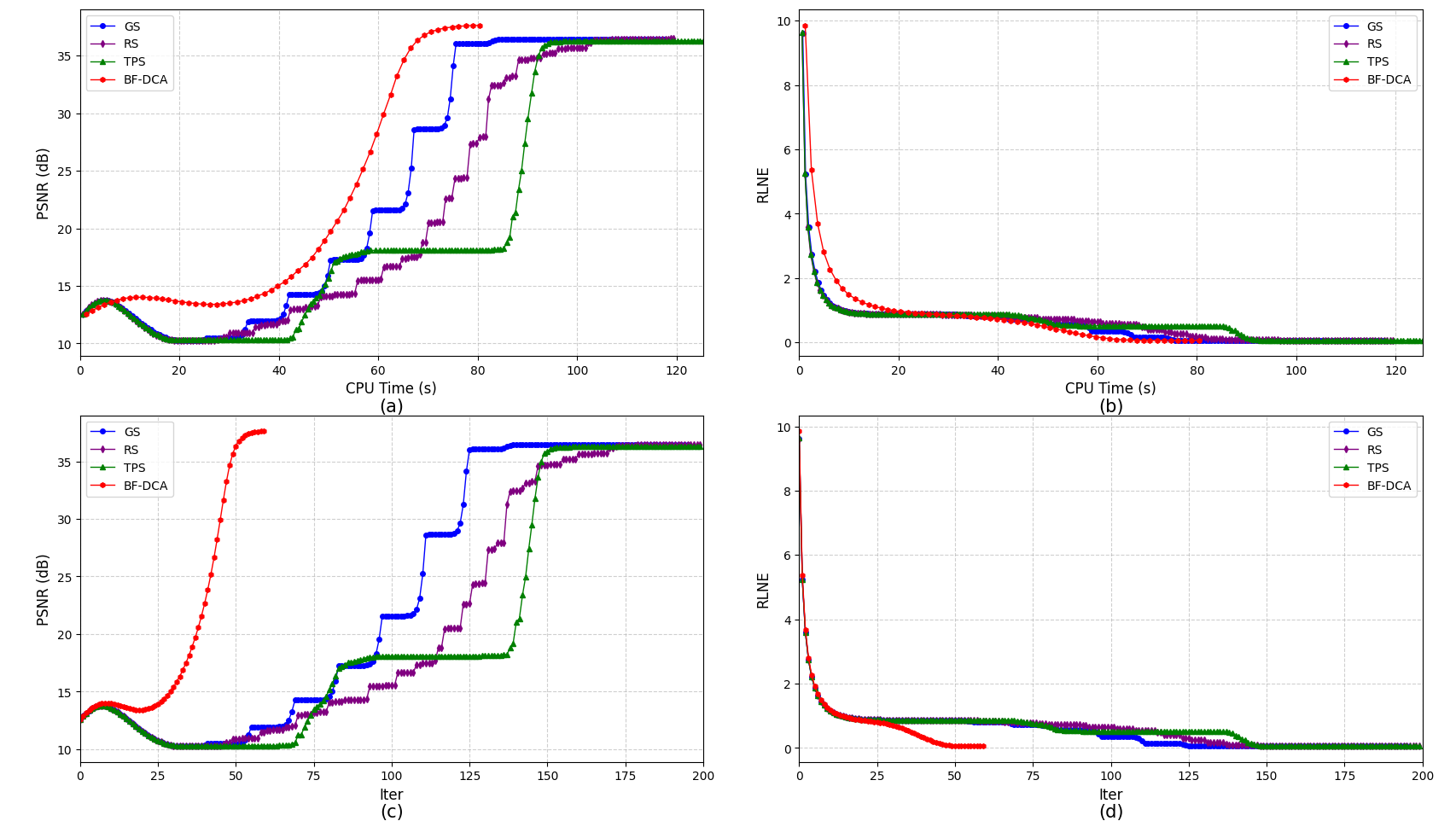}
	\vspace{0.1cm}
	\caption{Convergence behavior of all the tested algorithms on the phantom image: PSNR value versus computing time (a) and iterations (c); RLNE values versus computing time (b) and  iterations (d).}
	\label{fig:graph2}
\end{figure}
\newpage
In the third test, we further evaluate the performance of each algorithm using a brain magnetic resonance image. In this test, the parameter values for BF-DCA remain consistent with the previous settings. However, for GS, a hyperparameter search is conducted using a uniform grid of $10 \times 10$. In the case of RS, $100$ pairs of values for $(u_1, u_2)$ are selected randomly and uniformly. Meanwhile, TPE involves randomly selecting $100$ values of $(u_1, u_2)$.
The original image is displayed in (a) of Figure \ref{fig:graph3}, while its noisy undersampled version is presented in (b).
The recovered images by algorithms GS, RS, TPE, and BF-DCA are displayed in (c), (d), (e), and (f), respectively.
By comparing the images produced by each algorithm, we see that the image corresponding to BF-DCA displayed in (f) is closer to the ground truth (a), while the image in (e) is much farther from (a). This observation once again illustrates that the BF-DCA algorithm and combining with the bilevel programming approach \eqref{blp2}, is the best when compared to the other popular approaches.
Furthermore, we also plot the convergence curves regarding to the values of RLNE and PSNR as the iterative process in Figure  \ref{fig:graph4}.
From plots (b) and (d), we observe that all the tested algorithms can recover the brain image in terms of RLNE values, but BF-DCA performs the best. Additionally, from plots (a) and (c), we can see that our proposed algorithm, BF-DCA, significantly outperforms its competitors in terms of improving PSNR values.
Finally, to evidently illustrate the significant superiority of BF-DCA over other approaches, we list the detailed results of each algorithm with respect to the measurements Time, RLNE, and PSNR in Table \ref{tb1}.
We see from this table that, among all methods, the BF-DCA algorithm demonstrates superior performance on both tested images, achieving the shortest time, the lowest RLNE values, and the highest PSNR values. In summary, it can be concluded that the BF-DCA algorithm is the best in this comparison across both datasets.

\begin{figure}[htbp]
	\centering
	\includegraphics[width=0.16\textwidth]{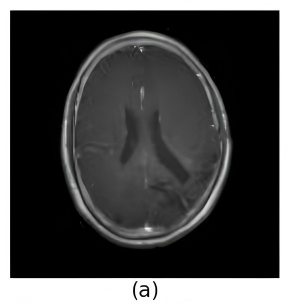}\hspace{-.1cm}
	\includegraphics[width=0.16\textwidth]{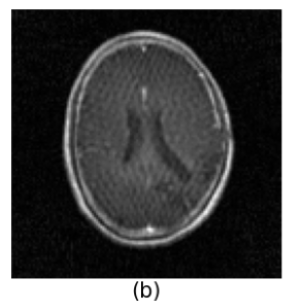}\hspace{-.1cm}
	\includegraphics[width=0.16\textwidth]{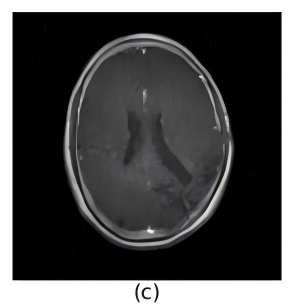}\hspace{-.1cm}
	\includegraphics[width=0.16\textwidth]{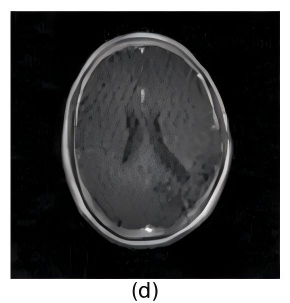}\hspace{-.1cm}
	\includegraphics[width=0.16\textwidth]{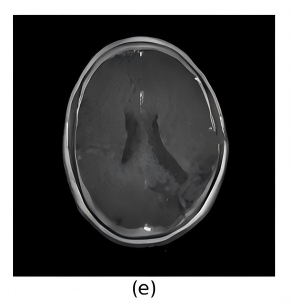}\hspace{-.1cm}
	\includegraphics[width=0.16\textwidth]{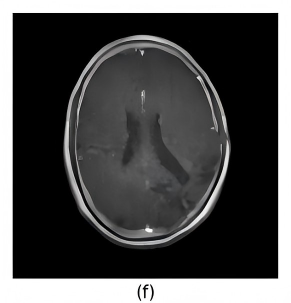}
	\vspace{-0.2cm}
	\caption{The ground-truth brain magnetic resonance image (a),  the undersampled image with added noise (b), and the restored images by algorithms GS (c), RS (d), TPE (e), and IR-DCA (f).}
	\label{fig:graph3}
\end{figure}
\begin{figure}[htbp]
	\centering
	\includegraphics[height=8cm]{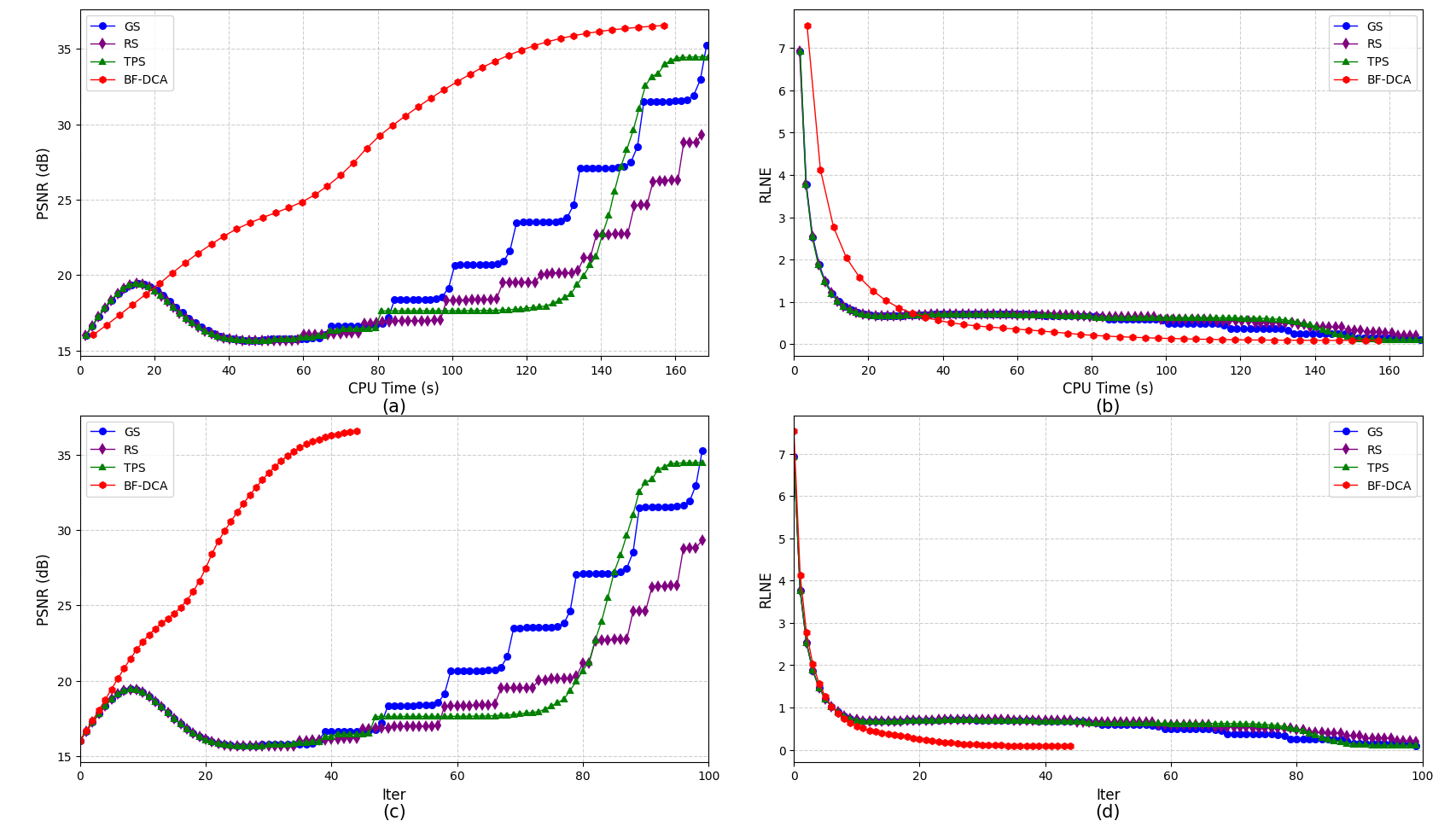}
	\vspace{0.1cm}
	\caption{Convergence behavior of all the tested algorithms on a brain magnetic resonance image: PSNR value versus computing time (a) and iterations (c); RLNE values versus computing time (b) and  iterations (d).}
	\label{fig:graph4}
\end{figure}

\begin{table}[htbp]
	\centering
	\caption{Detailed results of each algorithm on brain image and Shepp-Logan image.}
	\label{tb1}
	\small
	\begin{tabular}{llrrr}
		\hline
		{Dataset}&{Method} & {Time (s)} & {RLNE} & {PSNR (dB)} \\ \hline
		& GS & 168.366 & 0.101 & 35.251 \\
		Brain Image (128$\times$128)& RS & 167.063 & 0.199 & 29.322  \\
		& TPE & 168.873 & 0.110 & 34.469 \\
		&  {BF-DCA} & \textbf{156.976} & \textbf{0.086} & \textbf{36.554} \\\hline
		& GS & 117.541 & 0.061 & 36.466 \\
		Shepp-Logan Image (64$\times$64)& RS & 119.418 & 0.061 & 36.472 \\
		& TPE & 125.379 & 0.062 & 36.311 \\
		& {BF-DCA} & \textbf{80.425} & \textbf{0.054} & \textbf{37.629} \\\hline
	\end{tabular}
\end{table}

\subsection{Test on multiple images}

In this subsection, we use some magnetic resonance brain gray images to further illustrate the superiority of BF-DCA. As previous test, we also employ a pseudo-radial sampling scheme with a sampling rate of $57\%$. Furthermore, we also added Gaussian noise with a noise level of $0.001$.
In this experiment, for each dataset, we randomly select $100$ images of size $64 \times 64$ to form the dataset. Among them, $10$ images are used for training, $10$ images are used for validation, and $50$ images are used for testing. It is worth noting that all algorithms are evaluated on the same datasets.
For the GS, RS, and TPE algorithms, the hyperparameters  are selected based the interval $[-7,7]$, and all other parameter settings are kept the same as in the previous experiment.
For each image type, the experiments are repeated $10$ times under the same parameter settings.
For visibly see the see the images used in this test, we display $5$ brain images in Figure \ref{fig:graph5}.

\begin{figure}[htbp]
	\centering
	\includegraphics[width=1.0\textwidth]{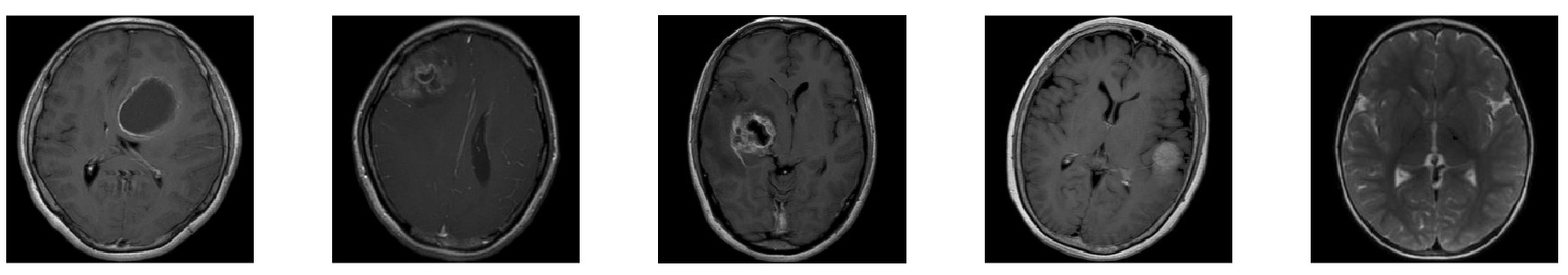} \\
	\vspace{-.5cm}
	\caption{Some brain gray images  used in this test.}
	\label{fig:graph5}
\end{figure}

For clearly distinct the performance of each algorithm, we report the detailed result in terms of  computing time (Time), and the validation error (Val.Err), and the testing error (Tes.Err) in Table \ref{tb3}.
From this table we see that,  the proposed BF-DCA algorithm consistently outperforms the competing methods. Specifically, BF-DCA achieves the shortest running time, reducing the computational cost by approximately $24\%$ compared with GS, RS, and TPE. Moreover, BF-DCA attains the lowest validation and test errors, demonstrating superior reconstruction accuracy and generalization performance. In contrast, GS and RS exhibit similar computational costs and error levels, while TPE shows relatively larger variability in validation and test errors. These results clearly indicate that BF-DCA provides a more efficient and accurate approach than the other methods under the same experimental settings.

\begin{table}[htbp]
	\centering
	\caption{Detailed results of algorithms on brain gray images.}
	\label{tb3}
	\begin{tabular}{lllll}
		\hline
		\textbf{Settings} & \textbf{Method} & \textbf{Time (s)} & \textbf{Val.Err.} & \textbf{Tes.Err.} \\ \hline
		$I_{\text{val}}=10$ & Grid Search & 173.994s \(\pm\) 2.685s & 0.563 \(\pm\) 0.019 & 0.557 \(\pm\) 0.008 \\
		$I_{\text{tes}}=10$ & Random Search & 174.293s \(\pm\)1.289s & 0.557 \(\pm\)0.018 &  0.552  \(\pm\) 0.007 \\
		$I_{\text{tra}}=50$ & TPE & 174.594s \(\pm\) 2.169s & 0.704 \(\pm\) 0.165 & 0.702 \(\pm\) 0.164 \\
		$I_{\text{Siz}}=64\times64$ & {BF-DCA} & \textbf{131.987s \(\pm\) 1.785s} & \textbf{0.441\(\pm\) 0.029} & \textbf{0.468 \(\pm\)0.017} \\
       	\hline
	\end{tabular}
\end{table}

%%%%%%%%%%%%%%%%%%%%%%%%%%%%%%%
\section{Conclusions}\label{sec5}
\setcounter{equation}{0}
%%%%%%%%%%%%%%%%%%%%%%%%%%%%%%%%%%%%%%%%%%%

In this paper, we investigated the automatic selection of hyperparameters for optimization-based image restoration models through a bilevel programming framework. By reformulating the original bilevel problem into an equivalent single-level problem with difference-of-convex inequality constraints, we developed an efficient algorithm that integrates ADMM for solving the lower-level problem with a feasibility penalty strategy and a proximal point scheme. The proposed method avoids the strong convexity assumptions commonly required by gradient-based bilevel approaches and is well suited for nonsmooth regularization models. We established global convergence of the generated sequence to a KKT stationary point of the equivalent problem. Extensive numerical experiments on both simulated and real images demonstrated that the proposed algorithm achieves superior restoration quality while requiring less computing time compared with several existing hyperparameter selection methods, such as, GS, RS, and TPE. These results indicate that the proposed framework provides a reliable and effective tool for automated hyperparameter selection in image restoration problems.

%\section*{Acknowledgements}
%We would like to thank the anonymous referees and the associate editor for their useful comments and suggestions
%which improved this paper greatly.

\section*{Disclosure statement}
The authors report there are no competing interests to declare.

\section*{Funding}

The work of Peili Li is supported by the National Natural Science Foundation of China (Grant No. 12301420), the Key Scientific Research Project of Universities in Henan Province (Grant No. 25A110002).
%The work of Yanyun Ding is supported by the Shenzhen Polytechnic University Research Fund (Grant No. 6024310021K).
The work of Qiuyu Wang is supported by the National Natural Science Foundation of China (Grant No. 12471307).

\bibliography{references}  %�ο����׿�������Ref

@article{lustig2007sparse,
	title={Sparse {MRI}: The application of compressed sensing for rapid MR imaging},
	author={Lustig, Michael and Donoho, David and Pauly, John M},
	journal={Magnetic Resonance in Medicine: An Official Journal of the International Society for Magnetic Resonance in Medicine},
	volume={58},
	number={6},
	pages={1182--1195},
	year={2007},
	publisher={Wiley Online Library}
}

@article{rudin1992nonlinear,
	title={Nonlinear total variation based noise removal algorithms},
	author={Rudin, Leonid I and Osher, Stanley and Fatemi, Emad},
	journal={Physica D: nonlinear phenomena},
	volume={60},
	number={1-4},
	pages={259--268},
	year={1992},
	publisher={Elsevier}
}

@inproceedings{ma2008efficient,
	title={An efficient algorithm for compressed {MR} imaging using total variation and wavelets},
	author={Ma, Shiqian and Yin, Wotao and Zhang, Yin and Chakraborty, Amit},
	booktitle={2008 IEEE Conference on Computer Vision and Pattern Recognition},
	pages={1--8},
	year={2008},
	organization={IEEE}
}

@article{yang2010fast,
	title={A fast alternating direction method for {TVL1-L2} signal reconstruction from partial Fourier data},
	author={Yang, Junfeng and Zhang, Yin and Yin, Wotao},
	journal={IEEE Journal of Selected Topics in Signal Processing},
	volume={4},
	number={2},
	pages={288--297},
	year={2010},
	publisher={IEEE}
}

@article{li2017two,
	title={Two-step fixed-point proximity algorithms for multi-block separable convex problems},
	author={Li, Qia and Xu, Yuesheng and Zhang, Na},
	journal={Journal of Scientific Computing},
	volume={70},
	number={3},
	pages={1204--1228},
	year={2017},
	publisher={Springer}
}

@article{ding2023efficient,
	title={Efficient dual ADMMs for sparse compressive sensing MRI reconstruction},
	author={Ding, Yanyun and Li, Peili and Xiao, Yunhai and Zhang, Haibin},
	journal={Mathematical Methods of Operations Research},
	volume={97},
	number={2},
	pages={207--231},
	year={2023},
	publisher={Springer}
}

@article{kunisch2013bilevel,
	title={A bilevel optimization approach for parameter learning in variational models},
	author={Kunisch, Karl and Pock, Thomas},
	journal={SIAM Journal on Imaging Sciences},
	volume={6},
	number={2},
	pages={938--983},
	year={2013},
	publisher={SIAM}
}

@incollection{dempe2020bilevel,
	title={Bilevel optimization: theory, algorithms, applications and a bibliography},
	author={Dempe, Stephan},
	booktitle={Bilevel optimization: advances and next challenges},
	pages={581--672},
	year={2020},
	publisher={Springer}
}

@inproceedings{franceschi2018bilevel,
	title={Bilevel programming for hyperparameter optimization and meta-learning},
	author={Franceschi, Luca and Frasconi, Paolo and Salzo, Saverio and Grazzi, Riccardo and Pontil, Massimiliano},
	booktitle={International conference on machine learning},
	pages={1568--1577},
	year={2018},
	organization={PMLR}
}

@article{feng2018gradient,
	title={Gradient-based regularization parameter selection for problems with nonsmooth penalty functions},
	author={Feng, Jean and Simon, Noah},
	journal={Journal of Computational and Graphical Statistics},
	volume={27},
	number={2},
	pages={426--435},
	year={2018},
	publisher={Taylor \& Francis}
}

@article{bertrand2022implicit,
	title={Implicit differentiation for fast hyperparameter selection in non-smooth convex learning},
	author={Bertrand, Quentin and Klopfenstein, Quentin and Massias, Mathurin and Blondel, Mathieu and Vaiter, Samuel and Gramfort, Alexandre and Salmon, Joseph},
	journal={Journal of Machine Learning Research},
	volume={23},
	number={149},
	pages={1--43},
	year={2022}
}

@inproceedings{shaban2019truncated,
	title={Truncated back-propagation for bilevel optimization},
	author={Shaban, Amirreza and Cheng, Ching-An and Hatch, Nathan and Boots, Byron},
	booktitle={The 22nd international conference on artificial intelligence and statistics},
	pages={1723--1732},
	year={2019},
	organization={PMLR}
}

@inproceedings{liu2018darts,
	author    = {Liu, Hanxiao and Simonyan, Karen and Yang, Yiming},
	title     = {{DARTS}: Differentiable Architecture Search},
	booktitle = {International Conference on Learning Representations},
	year      = {2019}
}

@article{ji2020convergence,
	title={Convergence of meta-learning with task-specific adaptation over partial parameters},
	author={Ji, Kaiyi and Lee, Jason D and Liang, Yingbin and Poor, H Vincent},
	journal={Advances in Neural Information Processing Systems},
	volume={33},
	pages={11490--11500},
	year={2020}
}

@inproceedings{ji2021bilevel,
	title={Bilevel optimization: Convergence analysis and enhanced design},
	author={Ji, Kaiyi and Yang, Junjie and Liang, Yingbin},
	booktitle={International conference on machine learning},
	pages={4882--4892},
	year={2021},
	organization={PMLR}
}

@inproceedings{raghu2019rapid,
	author    = {Raghu, Aniruddh and Raghu, Maithra and Bengio, Samy and Vinyals, Oriol},
	title     = {Rapid Learning or Feature Reuse? Towards Understanding the Effectiveness of {MAML}},
	booktitle = {International Conference on Learning Representations},
	year      = {2020}
}

@inproceedings{bertrand2020implicit,
	title={Implicit differentiation of lasso-type models for hyperparameter optimization},
	author={Bertrand, Quentin and Klopfenstein, Quentin and Blondel, Mathieu and Vaiter, Samuel and Gramfort, Alexandre and Salmon, Joseph},
	booktitle={International Conference on Machine Learning},
	pages={810--821},
	year={2020}
}

@inproceedings{grazzi2020iteration,
	title={On the iteration complexity of hypergradient computation},
	author={Grazzi, Riccardo and Franceschi, Luca and Pontil, Massimiliano and Salzo, Saverio},
	booktitle={International Conference on Machine Learning},
	pages={3748--3758},
	year={2020}
}

@inproceedings{gao2022value,
	title={Value function based difference-of-convex algorithm for bilevel hyperparameter selection problems},
	author={Gao, Lucy L and Ye, Jane and Yin, Haian and Zeng, Shangzhi and Zhang, Jin},
	booktitle={International conference on machine learning},
	pages={7164--7182},
	year={2022},
	organization={PMLR}
}

@article{ye2023difference,
	title={Difference of convex algorithms for bilevel programs with applications in hyperparameter selection},
	author={Ye, Jane J and Yuan, Xiaoming and Zeng, Shangzhi and Zhang, Jin},
	journal={Mathematical Programming},
	volume={198},
	number={2},
	pages={1583--1616},
	year={2023},
	publisher={Springer}
}

@article{2019A,
	title={A block symmetric {Gauss---Seidel} decomposition theorem for convex composite quadratic programming and its applications},
	author={ Li, Xudong  and  Sun, Defeng  and  Toh, Kim Chuan },
	journal={Mathematical Programming: Series A and B},
	volume={175},
	pages={395-418},
	year={2019}
}

@article{nesterov2014subgradient,
	title={Subgradient methods for huge-scale optimization problems},
	author={Nesterov, Yu},
	journal={Mathematical Programming},
	volume={146},
	number={1},
	pages={275--297},
	year={2014},
	publisher={Springer}
}

@article{attouch2013convergence,
	title={Convergence of descent methods for semi-algebraic and tame problems: proximal algorithms, forward--backward splitting, and regularized Gauss--Seidel methods},
	author={Attouch, Hedy and Bolte, J{\'e}r{\^o}me and Svaiter, Benar Fux},
	journal={Mathematical programming},
	volume={137},
	number={1},
	pages={91--129},
	year={2013},
	publisher={Springer}
}

@article{attouch2010proximal,
	title={Proximal alternating minimization and projection methods for nonconvex problems: An approach based on the Kurdyka-{\L}ojasiewicz inequality},
	author={Attouch, H{\'e}dy and Bolte, J{\'e}r{\^o}me and Redont, Patrick and Soubeyran, Antoine},
	journal={Mathematics of operations research},
	volume={35},
	number={2},
	pages={438--457},
	year={2010},
	publisher={INFORMS}
}

@article{liu2019refined,
	title={A refined convergence analysis of {pDCAe} with applications to simultaneous sparse recovery and outlier detection},
	author={Liu, Tianxiang and Pong, Ting Kei and Takeda, Akiko},
	journal={Computational Optimization and Applications},
	volume={73},
	number={1},
	pages={69--100},
	year={2019},
	publisher={Springer}
}

@book{clarke1990optimization,
	title={Optimization and nonsmooth analysis},
	author={Clarke, Frank H},
	year={1990},
	publisher={SIAM}
}

@book{rockafellar2015convex,
author    = {Rockafellar, R. Tyrrell},
title     = {Convex Analysis},
publisher = {Princeton University Press},
address   = {Princeton, NJ},
year      = {1970}
}

@inprocedings{B2013,
    title={Making a science of model search: Hyperparameter optimization in hundreds of dimensions for vision architectures},
    author={Bergstra, James and Yamins, Daniel and Cox, David},
    booktitle={International conference on machine learning},
    pages={115--123},
    year={2013},
    organization={PMLR}
}
\bibliographystyle{abbrv}      % 名字缩写

\end{document}